\documentclass[12pt]{amsart}
\usepackage{amsmath,amssymb,hyperref,geometry,graphicx,ytableau,fp,mathdots,amsthm,tikz,tikz-3dplot}
\usetikzlibrary{cd}
\usepackage{pgfplots}
\usepackage{float}
\usepgfplotslibrary{fillbetween}
\usepackage{theoremref}
\usetikzlibrary{arrows}

\newcounter{intro}[section]

\newtheorem{theorem}[equation]{Theorem}
\newtheorem{lemma}[equation]{Lemma}

\newtheorem{proposition}[equation]{Proposition}
\theoremstyle{definition}
\newtheorem{example}[equation]{Example}
\newtheorem{remark}[equation]{Remark}

\newtheorem*{Result}{Main Result}

\numberwithin{equation}{section}

\makeatletter
\def\moverlay{\mathpalette\mov@rlay}
\def\mov@rlay#1#2{\leavevmode\vtop{%
   \baselineskip\z@skip \lineskiplimit-\maxdimen
   \ialign{\hfil$\m@th#1##$\hfil\cr#2\crcr}}}
\newcommand{\charfusion}[3][\mathord]{
    #1{\ifx#1\mathop\vphantom{#2}\fi
        \mathpalette\mov@rlay{#2\cr#3}
      }
    \ifx#1\mathop\expandafter\displaylimits\fi}
\makeatother

\makeatletter
\newcommand*\bigcdot{\mathpalette\bigcdot@{.5}}
\newcommand*\bigcdot@[2]{\mathbin{\vcenter{\hbox{\scalebox{#2}{$\m@th#1\bullet$}}}}}
\makeatother

\newcommand*{\medcap}{\mathbin{\raisebox{1.8bp}{\resizebox{12bp}{8bp}{\ensuremath{\bigcap}}}}}
\newcommand{\dcap}{\charfusion[\mathbin]{\medcap}{*}}

\newcommand{\R}{\mathbb{R}}
\newcommand{\Z}{\mathbb{Z}}

\newcommand{\cF}{\mathcal{F}}

\newcommand{\Vol}{\mathrm{Vol}}
\newcommand{\vol}{\mathrm{vol}}
\newcommand{\cone}{\mathrm{cone}}
\newcommand{\rk}{\mathrm{rk}}
\newcommand{\I}{\mathcal{I}}
\newcommand{\J}{\mathcal{J}}
\newcommand{\M}{\mathcal{M}}
\newcommand{\cN}{\mathcal{N}}
\newcommand{\Inn}{\mathrm{Inn}}

\newcommand{\conv}{\mathrm{conv}}

\renewcommand{\phi}{\varphi}

\newcommand{\sM}{\mathsf{M}}
\newcommand{\cL}{\mathcal{L}}
\newcommand{\cl}{\mathrm{cl}}
\newcommand{\cG}{\mathcal{G}}

\newcommand{\cC}{\mathrm{Cub}}
\newcommand{\X}{\mathcal{X}}

\begin{document}
\DeclareRobustCommand{\subtitle}[1]{\\#1}
\title{{Tropical fans and normal complexes} \subtitle{\vspace{.25cm}\small{\normalfont{\emph{(Putting the ``volume'' back in ``volume polynomials'')}}}}}
\author[A.~Nathanson]{Anastasia Nathanson}
\address{Department of Mathematics, University of Minnesota}
\email{natha129@umn.edu}
\author[D.~Ross]{Dustin Ross}
\address{Department of Mathematics, San Francisco State University}
\email{rossd@sfsu.edu}

\begin{abstract}
Associated to any divisor in the Chow ring of a simplicial tropical fan, we construct a family of polytopal complexes, called normal complexes, which we propose as an analogue of the well-studied notion of normal polytopes from the setting of complete fans. We describe certain closed convex polyhedral cones of divisors for which the ``volume'' of each divisor in the cone---that is, the degree of its top power---is equal to the volume of the associated normal complexes. For the Bergman fan of any matroid with building set, we prove that there exists an open family of such cones of divisors with nonempty interiors. We view the theory of normal complexes developed in this paper as a polytopal model underlying the combinatorial Hodge theory pioneered by Adiprasito, Huh, and Katz.
\end{abstract}

\maketitle
\renewcommand{\subtitle}[1]{}

\vspace{-.5cm}
\section{Introduction}

In recent years, a compelling story has been unfolding wherein the main characters are special classes of noncomplete toric varieties masquerading as if they were smooth projective varieties. A notable plot point in this story is the work of Adiprasito, Huh, and Katz \cite{AHK}, who showed that Chow rings of noncomplete Bergman fans of matroids satisfy an analogue of the K\"ahler package. Their result has had significant impacts in combinatorics, solving decades-old log-concavity conjectures of Heron, Rota, and Welsch \cite{Rota,Heron,Welsh}, and it has led to a flurry of activity in ``combinatorial Hodge theory'' (see \cite{BES, BHMPW1, BHMPW2, ADH, AP1, AP2}, for example). 

Another combinatorial setting in which an analogue of the K\"ahler package arises is the polytope algebra of McMullen \cite{McMullen1}. For simple polytopes, McMullen's polytope algebra is isomorphic to the Chow ring of the corresponding projective toric variety \cite{McMullen2}, so one can view the polytope algebra as a type of polytopal model that underlies the algebro-geometric Hodge theory of projective toric varieties. Adiprasito, Huh, and Katz remark in \cite{AHK} that their proof of the K\"ahler package for general matroids was ``inspired by'' McMullen's proof of the analogous facts for polytope algebras, and this raises the question: \emph{Does there exist a polytopal model associated to Bergman fans of matroids that underlies the combinatorial Hodge theory developed by Adiprasito, Huh, and Katz?}

This paper introduces a new character to this story that we propose as the natural building block of a polytopal model for studying Chow rings of \emph{simplicial tropical fans}---a class of fans satisfying a weighted balancing condition and containing all Bergman fans of matroids. The new character that we introduce is the \emph{normal complex}, a polytopal complex associated to a noncomplete fan that generalizes the concept of normal polytopes associated to complete fans. \emph{The main result of this paper is that the degree of the top power of certain divisors in the Chow ring of a simplicial tropical fan is equal to the volume of the associated normal complex}, which is an analogue of a fundamental result in toric geometry regarding normal polytopes of complete fans. 

We view our result as a means by which one can import volume-theoretic tools and insights from polytopal geometry into the study of Chow rings of tropical fans. As an extension and application of these ideas, we mention that a recent paper of Lauren Nowak, Patrick O'Melveny, and the second author \cite{MNR} develops the theory of mixed volumes of normal complexes and proves an analogue of the Alexandrov--Fenchel inequalities in the normal complex setting; it turns out that the celebrated log-concavity of characteristic polynomials of matroids is then just a special case of these inequalities. 

The rest of the introduction gives an overview of the developments of this paper; we refer the reader to Section~\ref{sec:normalcomplexes} for precise definitions and a comprehensive discussion of these ideas. 

\subsection{Summary of results}

Let $\Sigma\subseteq N_\R$ be a simplicial tropical fan of dimension $d$ with associated degree function $\deg_\Sigma:A^d(\Sigma)\rightarrow\R$, and let $*\in\Inn(N_\R)$ be an inner product. Associated to $(\Sigma,*)$ and any value $z\in\R^{\Sigma(1)}$, we introduce a polytopal complex $C_{\Sigma,*}(z)$, called the \emph{normal complex}, which is obtained by truncating the cones of $\Sigma$ with hyperplanes that are normal to each ray---where ``normality'' is determined by $*$---and located a distance from the origin determined by $z$. The \emph{volume} of a normal complex, denoted $\Vol(C_{\Sigma,*}(z))$, is the sum of the volumes of its constituent $d$-dimensional polytopes.\footnote{We note that here and throughout the entire paper, volume will always be computed as simplicial volume, which is normalized so that a unit simplex has volume one.} There is a closed convex polyhedral cone $\overline \cC(\Sigma,*)\subseteq \R^{\Sigma(1)}$ comprised of $z$-values for which the truncating hyperplanes associated to the rays of each cone in $\Sigma$ intersect within that cone; we call these values \emph{pseudo-cubical}. Our main result can be stated as follows.

\begin{Result}
For each pseudo-cubical value $z\in\overline\cC(\Sigma,*)$, we have
\[
\deg_\Sigma(D^d)=\Vol(C_{\Sigma,*}(z)),
\] 
where $D\in A^1(\Sigma)$ is the divisor associated to $z$ under the quotient map $\R^{\Sigma(1)}\rightarrow A^1(\Sigma)$.
\end{Result}

We note that functions on divisors of the form $D\mapsto \deg(D^d)$ arise often in algebraic geometry, and they are generally called \emph{volume polynomials}. The terminology ``volume'' is motivated by the classical fact that, when $D$ is an ample divisor on a complex projective variety $\X$ of dimension $d$, the quantity $\deg(D^d)$ is the volume of $\X$ with respect to the K\"ahler metric associated to $D$. The term ``volume'' has also been proven apt in other ways; for example, in the setting of smooth complete toric varieties, the volume polynomial measures volumes of normal polytopes associated to nef divisors, and more generally, for smooth complete varieties that are not necessarily toric, the volume polynomial measures volumes of Newton--Okounkov bodies. For tropical fans, the (nontrivial) existence of a degree map allows us to define volume polynomials in an analogous way, but given that the most interesting tropical fans are not complete, none of the previously-studied volume-theoretic interpretations for volume polynomials are valid. Thus, the main result above may be viewed as a way of putting the ``volume'' back in ``volume polynomials'' of tropical fans.

\subsubsection{The construction of normal complexes}

We now outline the construction of normal complexes, which is closely related to and inspired by the construction of normal polytopes of complete fans. Let $\Sigma\subseteq N_\R$ be a simplicial fan of dimension $d$ and for each ray $\rho\in\Sigma(1)$, let $u_\rho\in N_\R$ be a distinguished ray generator. When $\Sigma$ is rational with respect to a lattice $N\subseteq N_\R$, we take $u_\rho\in N$ to be the primitive integral generator of $\rho$, but we do not generally assume that $\Sigma$ is rational. Given a divisor $D\in A^1(\Sigma)$, we can write $D$ (nonuniquely) as
\[
D=\sum_{\rho\in\Sigma(1)}z_\rho X_\rho
\]
where the sum is over the rays of $\Sigma$, each $X_\rho\in A^1(\Sigma)$ denotes the generator of the Chow ring associated to $\rho$, and the coefficients $z_\rho$ are real numbers. Each ray $\rho$ and coefficient $z_\rho$ corresponds to a half-space in the dual vector space $M_\R=N_\R^\vee$, defined by
\[
\{v\in M_\R\;|\;\langle v, u_\rho\rangle\leq z_\rho\}\subseteq M_\R,
\]
and the \emph{normal polyhedron} of $\Sigma$ associated to a choice of $z$-coefficients, denoted $P_\Sigma(z)\subseteq M_\R$, is the intersection of these half-spaces. Different choices of $z$-coefficients for the same divisor $D$ correspond to different translations of $P_\Sigma(z)$.

If $\Sigma$ is rational and complete, then $P_\Sigma(z)$ is the well-studied normal polytope of $D$, defined up to translation, and a fundamental result in toric geometry asserts that, when $D$ is nef, the volume of $P_\Sigma(z)$ is equal to $\deg_\Sigma(D^d)$. If $\Sigma$ is not complete, then there are two problems with this approach of simply computing the volume of $P_\Sigma(z)$:
\begin{enumerate}
\item the polyhedron $P_\Sigma(z)$ may be unbounded, so its volume may be infinite; and
\item even when $P_\Sigma(z)$ is bounded, its dimension is generally larger than $d$, so its volume---as a polynomial in $z$---will have degree larger than the volume polynomial.
\end{enumerate}
The construction of normal complexes, which requires the additional choice of an inner product $*\in\Inn(N_\R)$, remedies both of these issues. 

Given an inner product $*\in\Inn(N_\R)$, the \emph{normal complex} of $(\Sigma,*)$ with respect to a choice of coefficients $z\in\R^{\Sigma(1)}$ can be defined as
\[
C_{\Sigma,*}(z)=\Sigma\dcap P_\Sigma(z)\subseteq N_\R,
\]
where the notation $\dcap$ means that we take the intersection in $N_\R$ after using the inner product to identify $M_\R$ with $N_\R$.\footnote{Our definition of normal complexes in Section~\ref{sec:normalcomplexes} is slightly more technical than the one here, but this definition captures the intuitive idea and coincides with the proper definition for many values of $z$ and $*$.} As a polytopal complex, $C_{\Sigma,*}(z)$ can be thought of intuitively as a truncation of the cones of $\Sigma$ by normal hyperplanes; normality is determined by $*$ and the location of the hyperplanes is determined by $z$. While the shape of the normal complex and the volume of each of its constituent polytopes depend heavily on $*$ and the choice of $z$-coefficients for a given divisor $D$, a truly remarkable consequence of the main result above is that, so long as $\Sigma$ is tropical and $z$ is psuedo-cubical with respect to $(\Sigma,*)$, the total volume $\Vol(C_{\Sigma,*}(z))$ is independent of these choices and equal to $\deg_\Sigma(D^d)$. 

\subsubsection{Matroids and the pseudo-cubical hypothesis}

As was mentioned above, the psuedo-cubical hypothesis is the condition that the truncating hyperplanes associated to the rays of each cone in $\Sigma$ intersect within that cone. This condition is rather restrictive, and it is not clear from the outset whether the hypothesis of the main result above is nonvacuous for any interesting classes of tropical fans. To address this issue, we prove in Section~\ref{sec:matroids} that, if $\Sigma_{\sM,\cG}$ is the Bergman fan of a matroid $\sM$ with respect to any building set $\cG$, there is a nonempty open set in $\Inn(N_\R)$ for which the pseudo-cubical cone $\overline\cC(\Sigma_{\sM,\cG},*)$ has nonempty interior. This provides a large class of fans---fans that are of interest to both combinatorialists and algebraic geometers---for which the volume polynomial and the Chow ring can be studied using volume-theoretic tools from polytopal geometry. In particular, this class of Chow rings includes all Chow rings of wonderful compactifications associated to hyperplane arrangements \cite{DP1,DP2}, and the main result above sheds new light on the intersection theory of fundamental varieties in algebraic geometry, such as the moduli spaces $\overline\M_{0,n}$ of rational stable curves.

\subsection{Acknowledgements}

The authors are grateful to Federico Ardila, Matthias Beck, Emily Clader, Chris Eur, Serkan Ho\c{s}ten, and Leonid Monin for enlightening conversations related to this project. The second author would also like to thank Bernd Sturmfels who, while visiting San Francisco State University several years ago, posed the rather open-ended question: ``Can one describe the volume polynomial of $\overline\M_{0,n}$?" While there are many possible ways to answer this question, this work represents our favorite description (so far).

The second author was supported by a San Francisco State University Presidential Award during Fall 2020, and this research has also been supported by a grant from the National Science Foundation (RUI DMS-2001439).

\section{Normal complexes and their volumes}\label{sec:normalcomplexes}

In this section, we discuss the precise definitions, notations, motivations, and examples required for the development of normal complexes and their volumes. We view this section as an extended introduction that includes precise statements for all of the main results.

\subsection{Pure simplicial fans}

Let $M_\R$ and $N_\R$ be a dual pair of real vector spaces of dimension $n$, and denote the bilinear pairing by $\langle -, -\rangle$. Given a polyhedral fan $\Sigma\subseteq N_\R$, we denote the $k$-dimensional cones of $\Sigma$ by $\Sigma{(k)}$. Let $\preceq$ denote the face containment relation among the cones of $\Sigma$, and for each cone $\sigma\in\Sigma$, let $\sigma(k)\subseteq\Sigma(k)$ denote the $k$-dimensional faces of $\sigma$. For any cone $\sigma$ (or more generally, for any polyhedron $P$), let $\sigma^\circ$ (or $P^\circ$) denote the relative interior.

Henceforth, we adopt the convention that a \textbf{fan} $\Sigma\subseteq N_\R$ is a marked polyhedral fan, meaning that, in addition to specifying the polyhedral cones that comprise $\Sigma$, we have also chosen a distinguished generating vector $u_\rho\in \rho^\circ$ for each ray $\rho\in\Sigma(1)$. If $N\subseteq N_\R$ is a lattice---that is, a free abelian group such that $N_\R=N\otimes_\Z\R$---then we say that $\Sigma$ is \textbf{rational with respect to $N$} if each ray intersects the lattice at a nonzero vector. Given a fan $\Sigma\subseteq N_\R$ that is rational with respect to $N$, we always take $u_\rho$ to be the primitive integral generator of $\rho$---that is, $u_\rho$ is the first nonzero element of $N$ that lies on $\rho$. 

We say that a cone $\sigma$ is \textbf{simplicial} if $\dim(\sigma)=|\sigma(1)|$. Alternatively, simplicial cones are characterized by the property that their ray generators are linearly independent. Note that the faces of a simplicial cone $\sigma$ are in bijective correspondence with the subsets of $\sigma(1)$. For every face containment $\tau\preceq\sigma$ in a simplicial cone $\sigma$, let $\sigma\setminus\tau$ denote the face of $\sigma$ with rays $\sigma(1)\setminus\tau(1)$. If $\sigma$ is rational, then we say that $\sigma$ is \textbf{unimodular} if the primitive integral generators of any cone can be extended to a basis of $N$. Note that unimodular cones are simplicial. We say that a fan $\Sigma$ is \textbf{simplicial} or \textbf{unimodular} if every cone of $\Sigma$ is simplicial or unimodular. Every rational polyhedral fan $\Sigma$ determines a normal toric variety $\X_\Sigma$, and this variety is smooth if and only if $\Sigma$ is unimodular and has at worst finite quotient singularities if and only if $\Sigma$ is simplicial.

We say that a fan $\Sigma$ is \textbf{pure} if all of the maximal cones in $\Sigma$ have the same dimension. Henceforth, we assume that all fans are pure and we use the term \textbf{$d$-fan} to refer to a pure fan of dimension $d$.

\subsection{Chow rings}

Given a simplicial fan $\Sigma\subseteq N_\R$, the \textbf{Chow ring of $\Sigma$} is defined by
\[
A^\bullet(\Sigma)=\frac{\R\big[x_\rho\;|\;\rho\in\Sigma{(1)}\big]}{\I+\J}
\]
where
\[
\I=\big\langle x_{\rho_1}\cdots x_{\rho_k}\;|\;\cone(\rho_1,\dots,\rho_k)\notin\Sigma\big\rangle\;\;\;\text{ and }\;\;\;\J=\bigg\langle \sum_{\rho\in\Sigma{(1)}}\langle v,u_\rho\rangle x_\rho\;\bigg|\;v\in M_\R\bigg\rangle.
\]
If $\Sigma$ is unimodular, we note that $A^\bullet(\Sigma)$ is the Chow ring (in the usual intersection-theoretic sense) of the toric variety $\X_\Sigma$ (\cite{Danilov,BDP,Brion}). As both $\I$ and $\J$ are homogeneous, the Chow ring $A^\bullet(\Sigma)$ is a graded ring, and we denote by $A^k(\Sigma)$ the subgroup of homogeneous elements of degree $k$. We denote the generators of $A^\bullet(\Sigma)$ by $X_\rho=[x_\rho]\in A^1(\Sigma)$, and for any $\sigma\in\Sigma(k)$, we define
\[
X_\sigma=\prod_{\rho\in\sigma(1)}X_\rho\in A^k(\Sigma).
\]

\subsection{A guiding light: complete unimodular fans and normal polytopes}

Assume that $\Sigma$ is a unimodular fan that is also complete, meaning that every element of $N_\R$ is in some cone of $\Sigma$. This latter condition is equivalent to the condition that the corresponding toric variety $\X_\Sigma$ is complete in the algebro-geometric sense. In this setting, the algebro-geometric degree map is a linear isomorphism
\[
\deg_{\Sigma}:A^n(\Sigma)\rightarrow \R
\]
that is uniquely determined by linearity and the property that $\deg_\Sigma(X_\sigma)=1$ for all $\sigma\in\Sigma{(n)}$. Using the degree map, the \textbf{volume polynomial of $\Sigma$} is defined as the polynomial function
\begin{align*}
\Vol_\Sigma: A^1(\Sigma)&\rightarrow\R\\
D&\mapsto \deg_\Sigma(D^n).
\end{align*}
By definition, each divisor can be written (nonuniquely) as 
\begin{equation}\label{eq:divisorexpression}
D=\sum_{\rho\in\Sigma(1)}z_\rho X_\rho,
\end{equation}
and we often use these linear generators to view $\Vol_\Sigma$ as a homogeneous polynomial of degree $n$ in the variables $\{z_\rho\;|\;\rho\in\Sigma(1)\}$. Although the definition of the volume polynomial given above is purely algebraic, it also has a geometric interpretation, as we now describe.

Given a divisor $D\in A^1(\Sigma)$, presented as in \eqref{eq:divisorexpression}, define the \textbf{normal polytope of $\Sigma$ with respect to $z$} by
\[
P_\Sigma(z)=\big\{v\in M_\R\;|\;\langle v,u_\rho\rangle\leq z_\rho\text{ for all }\rho\in\Sigma{(1)}\big\}\subseteq M_\R.
\]
It follows from the definition of $\J$ that different choices of $z$ for the same divisor $D$ correspond to different translates of the same polytope. Let
\[
\Vol:\{\text{polytopes in }M_\R\}\rightarrow \R_{\geq 0}
\]
be the volume function that is normalized so that any fundamental simplex associated to the lattice $M=N^\vee\subseteq M_\R$ has unit volume. The guiding light for our work stems from a fundamental result in toric geometry \cite[Theorem~13.4.3]{Toric}, which asserts that, given any divisor $D=\sum z_\rho X_\rho$ for which the normal fan of $P_{\Sigma}(z)$ is refined by $\Sigma$---these correspond to nef divisors---we have
\begin{equation}\label{eq:toricvolume}
\Vol_\Sigma(D)=\Vol(P_\Sigma(z)).
\end{equation}
This beautiful result for complete fans is the primary motivation for our developments in the noncomplete setting. As such, we find it instructive to work out \eqref{eq:toricvolume} in a concrete example.

\begin{example}\label{ex:normalpolytope}
Let $N=\Z^2$ and let $\Sigma$ be the complete fan in $N_\R=\R^2$ depicted below.
\begin{center}
\begin{tikzpicture}[scale=2]
\draw[draw=blue!10,fill=blue!10, fill opacity=.5]    (.17,.07) -- (.9,.07) -- (.9,.8);
\draw[draw=blue!10,fill=blue!10, fill opacity=.5]    (.07,-.07) -- (.9,-.07) -- (.9,-.9) -- (.07,-.9);
\draw[draw=blue!10,fill=blue!10, fill opacity=.5]    (-.07,-.17) -- (-.07,-.9) -- (-.8,-.9);
\draw[draw=blue!10,fill=blue!10, fill opacity=.5 ]    (-.17,-.07) -- (-.9,-.07) -- (-.9,-.8);
\draw[draw=blue!10,fill=blue!10, fill opacity=.5]    (-.07,.07) -- (-.9,.07) -- (-.9,.9) -- (-.07,.9);
\draw[draw=blue!10,fill=blue!10, fill opacity=.5]    (.07,.2) -- (.07,.9) -- (.8,.9);
\draw[->] (0,0) -- (1,0); 
\draw[->] (0,0) -- (1,1); 
\draw[->] (0,0) -- (0,1); 
\draw[->] (0,0) -- (-1,0); 
\draw[->] (0,0) -- (-1,-1); 
\draw[->] (0,0) -- (0,-1); 
\node[right] at (1,0) {$\rho_1$};
\node[right] at (1,1) {$\rho_{12}$};
\node[right] at (0,1) {$\rho_2$};
\node[left] at (-1,0) {$\rho_{02}$};
\node[right] at (-1,-1) {$\rho_0$};
\node[right] at (0,-1) {$\rho_{01}$};
\end{tikzpicture}
\end{center}
The Chow ring of $\Sigma$ is
\[
A^\bullet(\Sigma)=\frac{\R[x_0,x_1,x_2,x_{01},x_{02},x_{12}]}{\I+\J}
\]
where $\I$ and $\J$ are described above. It can be checked from the definitions that
\begin{itemize}
\item $\deg_\Sigma(X_iX_{jk})=1$ if $i\in\{j,k\}$;
\item $\deg_\Sigma(X_i^2)=\deg_\Sigma(X_{ij}^2)=-1$;
\item the degree of any other quadratic monomial in the generators is zero.
\end{itemize}
Therefore, the volume polynomial is given by the following formula:
\begin{align*}
\Vol_\Sigma(z)&=2(z_0z_{01}+z_0z_{02}+z_1z_{01}+z_1z_{12}+z_2z_{02}+z_2z_{12})-(z_0^2+z_1^2+z_2^2+z_{01}^2+z_{02}^2+z_{12}^2).
\end{align*}

Using the dot product to identify $M_\R=N_\R$, we can draw the normal polytope associated to any specific $z$-value. If we choose the $z$-value carefully, then the original fan is the normal fan of the polytope $P_\Sigma(z)$; such as in the example depicted below.
\begin{center}
\begin{tikzpicture}[scale=1]
\draw[thick,fill=green!20, fill opacity=.5] (1,2) -- (2.5,.5) -- (2.5,-2) -- (-1,-2) -- (-2,-1) -- (-2,2) -- (1,2);
\draw[->] (0,0) -- (1,0); 
\draw[->] (0,0) -- (1,1); 
\draw[->] (0,0) -- (0,1); 
\draw[->] (0,0) -- (-1,0); 
\draw[->] (0,0) -- (-1,-1); 
\draw[->] (0,0) -- (0,-1); 
\node[right] at (1,0) {$\rho_1$};
\node[right] at (1,1) {$\rho_{12}$};
\node[right] at (0,1) {$\rho_2$};
\node[left] at (-1,0) {$\rho_{02}$};
\node[right] at (-1,-1) {$\rho_0$};
\node[right] at (0,-1) {$\rho_{01}$};
\node[rotate=270] at (2.8,-.8) {$x=z_1$};
\node[rotate=-45] at (2.1,1.5) {$x+y=z_{12}$};
\node[] at (-.5,2.3) {$y=z_{2}$};
\node[rotate=90] at (-2.3,.5) {$-x=z_{02}$};
\node[rotate=-45] at (-1.7,-1.7) {$-x-y=z_{0}$};
\node[] at (1,-2.3) {$-y=z_{01}$};
\end{tikzpicture}
\end{center}
Computing the simplicial area of this hexagon in terms of the $z$-coefficients, the reader should readily recover the formula for $\Vol_\Sigma(z)$, given above, verifying Formula \eqref{eq:toricvolume} in this example.
\end{example}

In the previous example, we made a specific choice of inner product on $N_\R$ (the standard dot product) in order to identify the vector spaces $M_\R$ and $N_\R$, which allowed us to draw the fan $\Sigma$ and the normal polytope $P_\Sigma(z)$ in the same vector space. While it was helpful to choose this inner product in order to draw a picture of $P_\Sigma(z)$, we note that this choice was not necessary in order to define $P_\Sigma(z)\subseteq M_\R$ or $\Vol(P_\Sigma(z))$.  As we will see in the next subsection, the situation is quite different in the noncomplete setting. In particular, when $\Sigma$ is not complete, the choice of an inner product is an essential ingredient in both the construction of normal complexes---which are analogues of normal polytopes in the noncomplete setting---and in the definition of their volume. In order to discuss these ideas in more detail, we now turn toward a discussion of noncomplete fans and their associated normal complexes.

\subsection{Noncomplete fans and normal complexes}

In this subsection, we introduce an analogue of normal polytopes---which we refer to as normal complexes---in the setting of noncomplete simplicial fans. Assume that $\Sigma$ is a (not-necessarily complete) simplicial $d$-fan in $N_\R$, and choose an inner product $*\in\Inn(N_\R)$. Normal complexes will be defined as polytopal complexes in $N_\R$ that depend on $(\Sigma,*)$, as well as on a value $z\in\R^{\Sigma(1)}$. Before defining normal complexes, we must describe the individual polytopes that comprise these polytopal complexes. 

Given a cone $\sigma\in\Sigma$, consider the polyhedron
\[
P_\sigma(z)=\big\{v\in M_\R\;|\;\langle v,u_\rho\rangle\leq z_\rho\text{ for all }\rho\in\sigma{(1)}\big\}\subseteq M_\R.
\]
The choice of inner product allows us to identify $N_\R$ with $M_\R$ via the natural isomorphism
\begin{align*}
N_\R&\rightarrow M_\R\\
u&\mapsto (u'\in N_\R\mapsto u* u'\in\R),
\end{align*}
and using this identification, we define polytopes
\[
P_{\sigma,*}(z)=\sigma\dcap P_\sigma(z).
\]
where the notation $\dcap$ means that we are intersecting $\sigma\subseteq N_\R$ with $P_\sigma(z)\subseteq M_\R$ after identifying the vector spaces $M_\R$ and $N_\R$ via the inner product $*$, as above. More explicitly, we have
\[
P_{\sigma,*}(z)=\sigma\cap\big\{v\in N_\R\;|\;v* u_\rho\leq z_\rho\text{ for all }\rho\in\sigma{(1)}\big\}\subseteq N_\R.
\]
The next example depicts these polytopes for the case of the complete fan of Example~\ref{ex:normalpolytope}.

\begin{example}\label{ex:normalpolytope2}
Consider the fan in Example~\ref{ex:normalpolytope} and let $*$ be the standard dot product. If we choose the $z$-values carefully---for example, if we use the same $z$-values that were chosen to draw the image in Example~\ref{ex:normalpolytope}---then the polytopes $P_{\sigma,*}(z)$ (and their faces) form a polytopal complex, depicted below, consisting of six quadrilaterals and their faces. Furthermore, the support of this polytopal complex is nothing more than the normal polytope $P_\Sigma(z)$, viewed as a subset of $N_\R$. 

\begin{center}
\begin{tikzpicture}[scale=1]
\draw[thick,fill=green!20, fill opacity=.5] (1,2) -- (2.5,.5) -- (2.5,-2) -- (-1,-2) -- (-2,-1) -- (-2,2) -- (1,2);
\draw[thick] (0,0) -- (2.5,0); 
\draw[thick] (0,0) -- (1.5,1.5); 
\draw[thick] (0,0) -- (0,2); 
\draw[thick] (0,0) -- (-2,0); 
\draw[thick] (0,0) -- (-1.5,-1.5); 
\draw[thick] (0,0) -- (0,-2); 
\end{tikzpicture}
\end{center}
If we're not so careful in how we choose the $z$-values---for example, if we decrease the value of $z_1$---then the polytopes $P_{\sigma,*}(z)$ no longer meet along faces, as we've depicted below, and their union is no longer equal to the normal polytope. 

\begin{center}
\begin{tikzpicture}[scale=1]
\draw[thick,fill=green!20, fill opacity=.5] (1,2) -- (1.5,1.5) -- (1,1) -- (1,-2) -- (-1,-2) -- (-2,-1) -- (-2,2) -- (1,2);
\draw[thick] (0,0) -- (1,0); 
\draw[thick] (0,0) -- (1.5,1.5); 
\draw[thick] (0,0) -- (0,2); 
\draw[thick] (0,0) -- (-2,0); 
\draw[thick] (0,0) -- (-1.5,-1.5); 
\draw[thick] (0,0) -- (0,-2); 
\end{tikzpicture}
\end{center}

\end{example}

As the previous example illustrates, if we want to define a polytopal complex using the polytopes $P_{\sigma,*}(z)$, then we require an extra compatibility between the inner product and the $z$-values in order to ensure that the polytopes $P_{\sigma,*}(z)$ meet along faces; we now introduce such a condition. We say that the value $z\in\R^{\Sigma(1)}$ is \textbf{cubical with respect to $(\Sigma,*)$} if, for all $\sigma\in\Sigma$, we have
\[
\sigma^\circ\cap\big\{v\in N_\R\;|\;v* u_{\rho}= z_{\rho}\text{ for all }\rho\in\sigma(1)\big\}\neq\emptyset,
\]
and we say that $z\in\R^{\Sigma(1)}$ is \textbf{pseudo-cubical with respect to $(\Sigma,*)$} if, for all $\sigma\in\Sigma$, we have
\[
\sigma\cap\big\{v\in N_\R\;|\;v* u_{\rho}= z_{\rho}\text{ for all }\rho\in\sigma(1)\big\}\neq\emptyset.
\]
Note that, because $\Sigma$ is simplicial, the intersections in these definitions contain at most one vector. Below, we have depicted what it means for the intersecting hyperplanes  to be cubical, pseudo-cubical, and neither in the case of a two-dimensional cone.
\begin{center}
\begin{tikzpicture}[scale=3]
\draw[draw=blue!10,fill=blue!10, fill opacity=.5]    (.1,.04) -- (.95,.04) -- (.95,.89);
\draw[thick,fill=green!20, fill opacity=.5] (0,0) -- (0.8,0) -- (0.8,0.4) -- (0.6,0.6);
\draw[->] (0,0) -- (1,0); 
\draw[->] (0,0) -- (1,1); 
\node[right] at (1,0) {$\rho_1$};
\node[right] at (1,1) {$\rho_{2}$};
\draw[thick] (0.8,-0.1) -- (0.8,0.55);
\draw[thick] (0.4,0.8) -- (0.9,0.3);
\node[] at (0.8,0.4) {$\bullet$};
\node[] at (0.5,-.2) {cubical};
\draw[draw=blue!10,fill=blue!10, fill opacity=.5]    (2.1,.04) -- (2.95,.04) -- (2.95,.89);
\draw[thick,fill=green!20, fill opacity=.5] (2,0) -- (2.6,0) -- (2.6,0.6);
\draw[->] (2,0) -- (3,0); 
\draw[->] (2,0) -- (3,1); 
\node[right] at (3,0) {$\rho_1$};
\node[right] at (3,1) {$\rho_{2}$};
\draw[thick] (2.6,-0.1) -- (2.6,0.8);
\draw[thick] (2.4,0.8) -- (2.9,0.3);
\node[] at (2.6,0.6) {$\bullet$};
\node[] at (2.5,-.2) {pseudo-cubical};
\draw[draw=blue!10,fill=blue!10, fill opacity=.5]    (4.1,.04) -- (4.95,.04) -- (4.95,.89);
\draw[thick,fill=green!20, fill opacity=.5] (4,0) -- (4.5,0) -- (4.5,0.5);
\draw[->] (4,0) -- (5,0); 
\draw[->] (4,0) -- (5,1); 
\node[right] at (5,0) {$\rho_1$};
\node[right] at (5,1) {$\rho_{2}$};
\draw[thick] (4.5,-0.1) -- (4.5,0.9);
\draw[thick] (4.35,0.85) -- (4.9,0.3);
\node[] at (4.5,0.7) {$\bullet$};
\node[] at (4.5,-.2) {not pseudo-cubical};
\end{tikzpicture}
\end{center}
In the cubical case of the two-dimensional setting depicted above, notice that the polytope $P_{\sigma,*}(z)$ is combinatorially equivalent to a square. In higher dimensions, we will see in Proposition~\ref{prop:cube} that $P_{\sigma,*}(z)$ is always combinatorially equivalent to a cube when $z$ is cubical, justifying the terminology.

As we will see in Proposition~\ref{prop:cubecone}, the set of cubical values forms an open convex polyhedral cone $\cC(\Sigma,*)\subseteq\R^{\Sigma(1)}$ and the set of pseudo-cubical values forms a closed convex polyhedral cone $\overline\cC(\Sigma,*)\subseteq\R^{\Sigma(1)}$ whose interior is $\cC(\Sigma,*)$. In Section~\ref{sec:cubical}, we also prove that, when $z\in\overline\cC(\Sigma,*)$ is pseudo-cubical, the polytopes $P_{\sigma,*}(z)$ do, in fact, meet along faces, implying that the collection of these polytopes and their faces forms a polytopal complex (Proposition~\ref{prop:polycomplex}). For a polyhedron $P$, let $\widehat P$ denote the polyhedral complex comprising all faces of $P$. For any pseudo-cubical $z\in\overline\cC(\Sigma,*)$, define the \textbf{normal complex of $\Sigma$ with respect to $z$ and $*$} as the polytopal complex
\begin{equation}\label{eq:normalcomplex}
C_{\Sigma,*}(z)=\bigcup_{\sigma\in\Sigma}\widehat P_{\sigma,*}(z).
\end{equation}

The next example depicts a normal complex in the noncomplete setting.

\begin{example}\label{ex:balancedfan}
Let $N_\R=\R^3$ and let $u_1,u_2,u_3$ be the standard basis vectors of $\R^3$. Set $u_0=-(u_1+u_2+u_3)$ and, for any subset $S\subseteq\{0,1,2,3\}$, define $u_S=\sum_{i\in S}u_i$. Let $\rho_S$ denote the ray spanned by $u_S$ and let $\Sigma$ be the two-dimensional fan depicted in the image below (for notational simplicity, we omit set brackets and commas for subsets $S\subseteq\{0,1,2,3\}$).

\begin{center}
\tdplotsetmaincoords{68}{55}
\begin{tikzpicture}[scale=2,tdplot_main_coords]
\draw[draw=blue!20,fill=blue!20,fill opacity=0.5]  (0,0, 0)-- (1, 0, 0) -- (1, 1, 1) -- cycle;
\draw[draw=blue!20,fill=blue!20,fill opacity=0.5]  (0,0, 0)-- (0, 1, 0) -- (1, 1, 1) -- cycle;
\draw[draw=blue!20,fill=blue!20,fill opacity=0.5]  (0,0, 0)-- (0, 0, 1) -- (1, 1, 1) -- cycle;
\draw[draw=blue!20,fill=blue!20,fill opacity=0.5] (0,0, 0)-- (1, 0, 0) -- (0, -1, -1) -- cycle;
\draw[draw=blue!20,fill=blue!20,fill opacity=0.5] (0,0, 0)-- (0, 1, 0) -- (-1, 0, -1) -- cycle;
\draw[draw=blue!20,fill=blue!20,fill opacity=0.5] (0,0, 0)-- (0, 0, 1) -- (-1, -1, 0) -- cycle;
\draw[draw=blue!20,fill=blue!20,fill opacity=0.5] (0,0, 0)-- (-1, -1, 0) -- (-1, -1, -1) -- cycle;
\draw[draw=blue!20,fill=blue!20,fill opacity=0.5] (0,0, 0)-- (-1, 0, -1) -- (-1, -1, -1) -- cycle;
\draw[draw=blue!20,fill=blue!20,fill opacity=0.5] (0,0, 0)-- (0, -1, -1) -- (-1, -1, -1) -- cycle;
\draw[->] (0,0,0) -- (1,0,0);
\node[right] at (1,0,0) {$\rho_1$};
\draw[->,gray] (0,0,0) -- (0,1,0); 
\node[above,gray] at (0,0.97,0) {$\rho_2$};
\draw[->] (0,0,0) -- (0,0,1);
\node[above] at (0,0,1) {$\rho_3$};
\draw[->] (0,0,0) -- (1,1,1);
\node[right] at (1,1,1) {$\rho_{123}$};
\draw[->] (0,0,0) -- (-1,-1,-1);
\node[left] at (-1,-1,-1) {$\rho_0$};
\draw[->] (0,0,0) -- (-1,-1,0);
\node[left] at (-1,-1,0) {$\rho_{03}$};
\draw[->,gray] (0,0,0) -- (-1,0,-1); 
\node[left,gray] at (-1,0,-1) {$\rho_{02}$};
\draw[->] (0,0,0) -- (0,-1,-1); 
\node[below] at (0,-1,-1) {$\rho_{01}$};
\end{tikzpicture}
\end{center}
In order to construct normal complexes, we require an inner product---let $*$ be the standard dot product on $\R^3$. The image below gives one example of a normal complex $C_{\Sigma,*}(z)$ with respect to one particular cubical value $z\in\cC(\Sigma,*)$---it is comprised of nine quadrilaterals and their faces. 

\begin{center}
    \tdplotsetmaincoords{68}{55}
\begin{tikzpicture}[scale=1.3,tdplot_main_coords]

\draw[draw=green!20, fill=green!20, fill opacity=.5] (0,0,0) -- (0, 0, 1.6) -- (1, 1, 1.6) -- (1.2, 1.2, 1.2) -- cycle; 
\draw[draw=green!20, fill=green!20, fill opacity=.5] (0,0,0) -- (0, 1.6, 0) -- (1, 1.6, 1) -- (1.2, 1.2, 1.2) -- cycle; 
\draw[draw=green!20, fill=green!20, fill opacity=.5] (0,0,0) -- (1.6, 0, 0) -- (1.6, 1, 1) -- (1.2, 1.2, 1.2) -- cycle; 
\draw[draw=green!20, fill=green!20, fill opacity=.5] (0,0,0) -- (0, 0, 1.6) -- (-1.6, -1.6, 1.6) -- (-1.6, -1.6, 0) -- cycle; 
\draw[draw=green!20, fill=green!20, fill opacity=.5] (0,0,0) -- (0, 1.6, 0) -- (-1.6, 1.6, -1.6) -- (-1.6, 0, -1.6) -- cycle; 
\draw[draw=green!20, fill=green!20, fill opacity=.5] (0,0,0) -- (1.6, 0, 0) -- (1.6, -1.6, -1.6) -- (0, -1.6, -1.6) -- cycle; 
\draw[draw=green!20, fill=green!20, fill opacity=.5] (0,0,0) -- (-1.6, -1.6, 0) -- (-1.6, -1.6, -.4) -- (-1.2, -1.2, -1.2) -- cycle; 
\draw[draw=green!20, fill=green!20, fill opacity=.5] (0,0,0) -- (-1.6, 0, -1.6) -- (-1.6, -.4, -1.6) -- (-1.2, -1.2, -1.2) -- cycle; 
\draw[draw=green!20, fill=green!20, fill opacity=.5] (0,0,0) -- (0, -1.6, -1.6) -- (-.4, -1.6, -1.6) -- (-1.2, -1.2, -1.2) -- cycle; 

\draw[thick] (0, 0, 1.6) -- (1, 1, 1.6) -- (1.2, 1.2, 1.2);
\draw[dashed] (0, 1.6, 0) -- (1, 1.6, 1) -- (1.2, 1.2, 1.2);
\draw[thick] (.7, 1.6, .7) -- (1, 1.6, 1) -- (1.2, 1.2, 1.2);
\draw[thick] (1.6, 0, 0) -- (1.6, 1, 1) -- (1.2, 1.2, 1.2);
\draw[thick] (0, 0, 1.6)  -- (-1.6,  -1.6, 1.6)  -- (-1.6, -1.6, 0);
\draw[dashed] (0, 1.6, 0) -- (-1.6,  1.6, -1.6) -- (-1.6, 0, -1.6);
\draw[thick] (1.6, 0, 0) -- (1.6, -1.6, -1.6) -- (0, -1.6, -1.6);
\draw[thick] (-1.6, -1.6, 0) -- (-1.6, -1.6, -.4) -- (-1.2, -1.2, -1.2);
\draw[dashed] (-1.6, 0, -1.6) -- (-1.6, -.4, -1.6) -- (-1.2, -1.2, -1.2);
\draw[thick] (0, -1.6, -1.6) -- (-.4, -1.6, -1.6) -- (-1.2, -1.2, -1.2);

\draw[thick] (0, 0, 0) -- (1.2, 1.2, 1.2);
\draw[thick] (0, 0, 0) -- (0, 0, 1.6);
\draw[dashed] (0, 0, 0) -- (0, 1.6, 0);
\draw[thick] (0, 0, 0) -- (1.6, 0, 0);
\draw[thick] (0, 0, 0) -- (0, -1.6, -1.6);
\draw[dashed] (0, 0, 0) -- (-1.6, 0, -1.6);
\draw[thick] (0, 0, 0) -- (-1.6, -1.6, 0);
\draw[thick] (0, 0, 0) -- (-1.2, -1.2, -1.2);
\end{tikzpicture}
\end{center}
Changing the $z$-values corresponds to sliding the boundary components of the normal complex along the corresponding rays of $\Sigma$, and the cubical $z$-values correspond to those deformations for which the combinatorial structure of the polytopal complex is constant. 
\end{example}

\begin{remark}
As mentioned in the introduction, one could alternatively define the notion of a normal complex of $(\Sigma,*)$ with respect to $z$  as
\begin{equation}\label{eq:altnormalcomplex}
\Sigma\dcap P_\Sigma(z)
\end{equation}
where $P_\Sigma(z)$ is the polyhedron
\[
P_\Sigma(z)=\big\{v\in M_\R\;|\;\langle v,u_\rho\rangle\leq z_\rho\text{ for all }\rho\in\Sigma{(1)}\big\}\subseteq M_\R.
\]
This alternative definition certainly has advantages; for example, this approach does not require the pseudo-cubical condition as part of the definition and yields a polytopal complex for \emph{any} $z$-value. Moreover, in the setting of complete fans, the support of this polytopal complex can always be identified with the normal polytope, so \eqref{eq:altnormalcomplex} is a true generalization of normal polytopes to the noncomplete setting.

To justify why we have opted not to use this alternative approach, first observe that definitions \eqref{eq:normalcomplex} and \eqref{eq:altnormalcomplex} agree whenever $z\in\overline\cC(\Sigma,*)$ \emph{and} 
\begin{equation}\label{eq:altcond}
P_{\sigma,*}(z)\subseteq \big\{v\in N_\R\;|\;v* u_\rho\leq z_\rho\text{ for all }\rho\notin\sigma{(1)}\big\}.
\end{equation}
While the pseudo-cubical condition can be checked locally cone-by-cone, the extra condition \eqref{eq:altcond} is rather cumbersome to work with, requiring an understanding of the global geometry of $\Sigma$. The reason we have chosen to work with the slightly more technical definition \eqref{eq:normalcomplex} instead of the more straightforward definition \eqref{eq:altnormalcomplex} is essentially so we do not require the extra condition \eqref{eq:altcond} as a hypothesis for our results. If we include this hypothesis, then our results apply to both definitions, but using the approach in \eqref{eq:normalcomplex} allows us to prove these results for a more general set of $z$-values.
\end{remark}

\begin{remark}
For a given fan $\Sigma\subseteq N_\R$ with inner product $*\in\Inn(N_\R)$, it can be shown that every pseudo-cubical value gives rise to a \emph{convex} piecewise linear map on $\Sigma$, where convexity is in the sense of \cite[Definition~4.1]{AHK}. In particular, if $\Sigma$ is complete and unimodular, pseudo-cubical values give rise to nef divisors on the associated toric variety. On the other hand, it is not hard to find examples of complete, unimodular fans with a fixed inner product that admit nef divisors that cannot be represented by pseudo-cubical values. In other words, in the complete, unimodular setting, not every normal polytope can be represented as the support of a normal complex, so our results do not strictly generalize \eqref{eq:toricvolume}. However, the methods in this paper imply that our volume-theoretic interpretation of the volume polynomial can be extended to all $z$-values as long as one is willing to work with signed volumes of simplices, and it then follows from a recent result of Schneider \cite[Proposition~1]{Schneider2} that this more general interpretation does, indeed, generalize \eqref{eq:toricvolume} for all convex values.
\end{remark}

\subsection{Volumes of normal complexes}

We now discuss how to define volumes of normal complexes. As in the case of complete fans, we should normalize volumes of polytopes using dual lattices. However, since each polytope $P_{\sigma,*}(z)$ lies in a subspace of $N_\R$, some additional care must be taken in order to define the appropriate normalization. 

For each cone $\sigma\in\Sigma$, define the subgroup
\[
N_\sigma=\mathrm{span}_\Z(u_\rho\;|\;\rho\in\sigma(1))\subseteq N_\R,
\]
and let $M_\sigma$ denote the dual of $N_\sigma$. Using the inner product $*$, we can identify $M_{\sigma,\R}=M_\sigma\otimes\R$ with $N_{\sigma,\R}=N_\sigma\otimes\R$ and thus, we can view $M_\sigma$ as a lattice in $N_{\sigma,\R}$. For each $\sigma\in\Sigma$, let 
\[
\Vol_\sigma:\big\{\text{polytopes in }N_{\sigma,\R}\big\}\rightarrow\R_{\geq 0}
\]
be the volume function determined by the property that a fundamental simplex of the lattice $M_\sigma\subseteq N_{\sigma,\R}$ has unit volume, and define the \textbf{volume of the normal complex $C_{\Sigma,*}(z)$} as the sum of the volumes of the constituent $d$-dimensional polytopes:
\[
\Vol( C_{\Sigma,*}(z))=\sum_{\sigma\in\Sigma(d)}\Vol_\sigma(P_{\sigma,*}(z)).
\]
In slightly more generality, suppose that $\omega:\Sigma(d)\rightarrow \R_{>0}$ is a weight function on the maximal cones of $\Sigma$. The \textbf{volume of the normal complex $C_{\Sigma,*}(z)$ weighted by $\omega$} is defined by
\begin{equation}\label{eq:weightedvolume}
\Vol(C_{\Sigma,*}(z);\omega)=\sum_{\sigma\in\Sigma(d)}\omega(\sigma)\Vol_\sigma(P_{\sigma,*}(z)).
\end{equation}

One of our main results regarding normal complexes of general simplicial fans is an explicit computation of their volume. In Theorem~\ref{thm:volumecomputation}, we prove that, for every $z\in\overline\cC(\Sigma,*)$ and $\sigma\in\Sigma(k)$, we have
\begin{equation}\label{eq:geomvolumeformula}
\Vol_\sigma(P_{\sigma,*}(z))=\det(G_\sigma)\sum_{f\in L(\sigma)}\prod_{j=1}^k\frac{\det(G_{\sigma(f,j),f(j)}(z))}{\det(G_{\sigma(f,j)})}
\end{equation}
where the notation is defined as follows:
\begin{itemize}
\item for $\sigma\in\Sigma(k)$, the set $L(\sigma)$ is the set of bijections $f:\{1,\dots,k\}\rightarrow\sigma(1)$;
\item for $f\in L(\sigma)$ and $1\leq j\leq k$, the cone $\sigma(f,j)\preceq\sigma$ has rays indexed by $\{{f(i)}\;|\;i\leq j\}$;
\item the matrix $G_\sigma$ is defined by $G_\sigma=(u_\rho*u_\eta)_{\rho,\eta\in\sigma(1)}$;
\item the matrix $G_{\sigma,\rho}(z)$ is obtained by replacing the $\rho$th column of $G_\sigma$with $z_\sigma=(z_\eta)_{\eta\in\sigma(1)}$.
\end{itemize}
As we will see in Section~\ref{sec:computingvolumes}, this formula for $\Vol_\sigma(P_{\sigma,*}(z))$ follows from a specific triangulation of $P_{\sigma,*}(z)$ that we describe explicitly in Proposition~\ref{prop:triangulate}.

If $\Sigma$ happens to be a complete unimodular fan, then it is not hard to see from the definitions that volumes of normal complexes reduces to volumes of normal polytopes:
\[
\Vol(C_{\Sigma,*}(z))=\Vol(P_\Sigma(z)) 
\]
for all $z\in\overline\cC(\Sigma,*)$. In particular, $\Vol(C_{\Sigma,*}(z))$ is independent of the choice of inner product when $\Sigma$ is complete. When $\Sigma$ is not complete, however, one should not expect volumes of normal complexes to be independent of this choice. The next example illustrates how the choice of the inner product $*$ influences the shape of normal complexes as well as the computation of their volumes.

\begin{example}\label{ex:depends}
Let $\Sigma$ be the fan associated to the first quadrant in $N_\R=\R^2$:
\begin{center}
\begin{tikzpicture}[scale=3]
\draw[draw=blue!10,fill=blue!10, fill opacity=.5]    (.05,.05) -- (.95,.05) -- (.95,.95) -- (.05,.95);
\draw[->] (0,0) -- (1,0); 
\draw[->] (0,0) -- (0,1); 
\node[right] at (1,0) {$\rho_1$};
\node[left] at (0,1) {$\rho_2$};
\node[] at (0.5,0.5) {$\sigma$};
\node[] at (-.2,.5) {$\Sigma\;=$};
\end{tikzpicture}
\end{center}
Let $z=(z_1,z_2)=(2,2)$ and let $*=\bigcdot$ be the standard dot product. Then the polytope $P_{\sigma,\bigcdot}(2,2)$ is the $2\times 2$ square depicted in the image below.
\begin{center}
\begin{tikzpicture}[scale=1.5]
\draw[thick,fill=green!20, fill opacity=.5] (0,0) -- (2,0) -- (2,2) -- (0,2) -- (0,0);
\draw[dashed] (0,0) -- (1,0) -- (0,1) -- (0,0);
\node at (0,0) {$\bullet$};
\node at (1,0) {$\bullet$};
\node at (2,0) {$\bullet$};
\node at (0,1) {$\bullet$};
\node at (1,1) {$\bullet$};
\node at (2,1) {$\bullet$};
\node at (0,2) {$\bullet$};
\node at (1,2) {$\bullet$};
\node at (2,2) {$\bullet$};
\node[below] at (1,0) {$u_1$};
\node[left] at (0,1) {$u_2$};
\node[] at (-.12,-.12) {$0$};
\end{tikzpicture}
\end{center}
In this image, we have also included a part of the lattice $M_\sigma$, along with a fundamental simplex. From this picture, we see that $\Vol_{\sigma}(P_{\sigma,\bigcdot}(2,2))=8$. 

We could just as well choose a different inner product; for example, let us consider the inner product $*=\star$ defined by
\[
(a,b)\star(c,d)=4ac+ad+bc+2bd.
\]
Using the same choice $z=(z_1,z_2)=(2,2)$, we have depicted the polytope $P_{\sigma,\star}(2,2)$ below, along with a part of the lattice $M_\sigma$ and a fundamental simplex.
\begin{center}
\begin{tikzpicture}[scale=3]
\draw[thick,fill=green!20, fill opacity=.5] (0,0) -- (1/2,0) -- (2/7,6/7) -- (0,1) -- (0,0);
\draw[dashed] (0,0) -- (2/7,-1/7)  -- (1/7,3/7)  -- (0,0);
\node[below] at (1,0) {$u_1$};
\node[left] at (0,1) {$u_2$};
\node[] at (-.06,-.06) {$0$};
\node at (0,0) {$\bullet$};
\node at (0,1) {$\bullet$};
\node at (1,0) {$\bullet$};
\node at (-1/7,4/7) {$\bullet$};
\node at (1/7,3/7) {$\bullet$};
\node at (3/7,2/7) {$\bullet$};
\node at (5/7,1/7) {$\bullet$};
\node at (2/7,6/7) {$\bullet$};
\node at (4/7,5/7) {$\bullet$};
\node at (6/7,4/7) {$\bullet$};
\node at (-2/7,1/7) {$\bullet$};
\node at (2/7,-1/7) {$\bullet$};
\end{tikzpicture}
\end{center}
By chopping up the fundamental simplex and filling the polytope, we can see that 
\[
\Vol_{\sigma}(P_{\sigma,\star}(2,2))=5\neq 8= \Vol_{\sigma}(P_{\sigma,\bigcdot}(2,2)).
\]
Since $\Sigma$ contains just a single $2$-dimensional cone, we have $\Vol(C_{\Sigma,*}(z))=\Vol_{\sigma}(P_{\sigma,*}(z))$ for any $z$ and $*$, from which we see that the volumes of the normal complexes associated to this noncomplete fan $\Sigma$ depend in a nontrivial way on the choice of inner product.
\end{example}

Example~\ref{ex:depends} illustrates that $\Vol(C_{\Sigma,*}(z))$ depends nontrivially on the choice of $*$; however, one might be so optimistic as to hope that there is a nice family of noncomplete fans that shares a particular type of symmetry for which weighted volumes of normal complexes are independent of the choice of inner product. As we will see below, independence of $*$ will naturally and directly lead us to the concept of tropical fans.

\subsection{Square-free expressions}

The expression in the right-hand side of Equation~\eqref{eq:geomvolumeformula} also arises in a natural way when computing products of divisors in $A^\bullet(\Sigma)$. Suppose that $\Sigma\subseteq N_\R$ is a simplicial fan and, for all $z\in\R^{\Sigma(1)}$, denote
\[
D(z)=\sum_{\rho\in\Sigma(1)}z_\rho X_\rho.
\]
Given any inner product $*\in\Inn(N_\R)$, we prove in Theorem~\ref{thm:squarefree} that
\begin{equation}\label{eq:squarefree}
D(z_1)\cdots D(z_k)=\sum_{\sigma\in\Sigma(k)}\Big(\det(G_\sigma)\sum_{f\in L(\sigma)}\prod_{j=1}^k\frac{\det(G_{\sigma(f,j),f(j)}(z_j))}{\det(G_{\sigma(f,j)})}\Big)X_\sigma,
\end{equation}
where all notation is as in \eqref{eq:geomvolumeformula}. Equation~\eqref{eq:squarefree} provides a way of expressing any product of divisors in $A^\bullet(\Sigma)$ as a linear combination of square-free products of divisors. Moreover, we also prove that the coefficients in the right-hand side of \eqref{eq:squarefree} are positive if $z_1,\dots,z_k\in\cC(\Sigma,*)$ and nonnegative if $z_1,\dots,z_k\in\overline\cC(\Sigma,*)$, so this result provides a way of computing \emph{effective} square-free expressions of pseudo-cubical divisors.

\subsection{Degree maps, tropical fans, and volume polynomials}

From \eqref{eq:weightedvolume}, \eqref{eq:geomvolumeformula}, and \eqref{eq:squarefree}, we see that the weighted volume $\Vol(C_{\Sigma,*}(z);\omega)$ is the weighted sum of the coefficients of $D(z)^d\in A^d(\Sigma)$, as long as we express $D(z)^d$ using the square-free formula in \eqref{eq:squarefree}, which depends on $*\in\Inn(N_\R)$. Therefore, in order to determine whether $\Vol(C_{\Sigma,*}(z);\omega)$ is independent of $*$, it suffices to know whether there exists a well-defined linear degree map
\begin{equation}\label{eq:tropicaldegreemap}
\deg_{\Sigma,\omega}:A^d(\Sigma)\rightarrow\R
\end{equation}
such that $\deg_{\Sigma,\omega}(X_\sigma)=\omega(\sigma)$ for every $\sigma\in\Sigma(d)$. In fact, an elementary computation \cite[Proposition~5.6]{AHK} shows that such a degree map exists if and one if
\begin{equation}\label{eq:weighted}
\sum_{\sigma\in\Sigma(k)\atop \tau\preceq\sigma}\omega(\sigma)u_{\sigma\setminus\tau}\in N_{\tau,\R}\;\;\;\text{ for every }\;\;\;\tau\in\Sigma(d-1).
\end{equation}

The weighted balancing condition in \eqref{eq:weighted} is the defining property of a tropical fan. More precisely, a \textbf{tropical $d$-fan} $(\Sigma,\omega)$ is a $d$-fan $\Sigma$ along with a weight function $\omega:\Sigma(d)\rightarrow\R_{>0}$ satisfying \eqref{eq:weighted}. When $\omega(\sigma)=1$ for all $\sigma$, we say that $\Sigma$ is \textbf{balanced}, and we omit $\omega$ from the notation. Given a simplicial tropical fan $(\Sigma,\omega)$, we define the \textbf{volume polynomial} by
\begin{align*}
\Vol_{\Sigma,\omega}:A^1(\Sigma)&\rightarrow\R\\
D&\mapsto\deg_{\Sigma,\omega}(D^d),
\end{align*}
where the tropical degree map $\deg_{\Sigma,\omega}$ is determined by the property that $\deg_{\Sigma,\omega}(X_\sigma)=\omega(\sigma)$ for every $\sigma\in\Sigma(d)$.

Our main result (Theorem~\ref{thm:main}) can now be stated precisely. Let $(\Sigma,\omega)$ be a simplicial tropical fan in $N_\R$ and choose an inner product $*\in\Inn(N_\R)$. For any $D\in A^1(\Sigma)$ and $z\in\overline\cC(\Sigma,*)$ with $D=\sum z_\rho X_\rho$, we have
\begin{equation}\label{eq:alg=geom}
\Vol_{\Sigma,\omega}(D)=\Vol(C_{\Sigma,*}(z);\omega).
\end{equation}
Remarkably, even though the shape and volume of each polytope in $C_{\Sigma,*}(z)$ depends nontrivially on the choice of inner product and $z$-coordinates used to represent $D$, the weighted sum of volumes of these polytopes is independent of these choices.

\section{The cubical hypothesis}\label{sec:cubical}

In this section, we develop a number of preparatory results regarding normal complexes. These results will be especially important when it comes time to compute volumes of normal complexes in the next section. Throughout this section, let $\Sigma\subseteq N_\R$ denote a simplicial $d$-fan, and let $*\in\Inn(N_\R)$ be an inner product.

\subsection{Recharacterizing the cubical hypothesis}

In this subsection, we introduce a useful characterization of the (pseudo-)cubical hypothesis. Let $z\in\R^{\Sigma(1)}$ and for each cone $\sigma\in\Sigma$, define $w_\sigma$ to be the unique vector in the following intersection
\[
N_{\sigma,\R}\cap\big\{v\in N_\R\;|\;v* u_{\rho}= z_{\rho}\text{ for all }\rho\in\sigma(1)\big\}=\{w_\sigma\}.
\]
The fact that the intersection contains a single vector follows from the assumption that $\Sigma$ is simplicial. Given a cone $\sigma\in\Sigma$, we say that the value $z\in\R^{\Sigma(1)}$ is \textbf{cubical (pseudo-cubical) with respect to $(\sigma,*)$} if
\[
w_\tau\in\tau^\circ\;\;\;(w_\tau\in\tau)\;\;\;\text{ for all faces }\;\;\;\tau\preceq\sigma.
\]
Notice that $z\in\R^{\Sigma(1)}$ is (pseudo-)cubical with respect to $(\Sigma,*)$ (as defined in the previous section) if and only if it is (pseudo-)cubical  with respect to $(\sigma,*)$ for each $\sigma\in\Sigma(d)$. The next results provides an alternative characterization of the (pseudo-)cubical hypothesis.

\begin{proposition}\label{prop:charcub} 
A value $z\in\R^{\Sigma{(1)}}$ is cubical (pseudo-cubical) with respect to $(\sigma,*)$ if and only if
\[
u_\rho* w_\tau<z_\rho\;\;\;(u_\rho* w_\tau\leq z_\rho)\;\;\;\text{ for all faces }\;\;\;\tau \preceq \sigma\;\;\;\text{ and rays }\;\;\;\rho\notin \tau(1).
\]
\end{proposition}

\begin{proof}
Given a face $\tau\preceq \sigma$ and a ray $\rho\in\tau(1)$, the subspace $N_{\tau\setminus\rho,\R}$ divides $N_{\tau,\R}$ into two half-spaces; let $H_{\tau,\rho}$ denote the closed half-space that contains $u_{\rho}\notin N_{\tau\setminus\rho,\R}$. The cone $\tau$ has a half-space presentation
\[
\tau=N_{\tau,\R}\cap\bigcap_{\rho\in\tau(1)} H_{\tau,\rho}\;\;\;\Longrightarrow\;\;\;\tau^\circ=N_{\tau,\R}\cap\bigcap_{\rho\in\tau(1)} H_{\tau,\rho}^\circ.
\]

We claim that $w_\tau\in H_{\tau,\rho}^\circ$ if and only if $u_{\rho}* w_{\tau\setminus\rho}<z_{\rho}$. This follows from the following three observations.
\begin{enumerate}
\item Since $w_{\tau\setminus\rho}\in N_{\tau\setminus\rho,\R}$, then $w_\tau\in H_{\tau,\rho}^\circ$ if and only if $w_\tau-w_{\tau\setminus\rho}\in H_{\tau,\rho}^\circ$. 
\item For all $\rho'\in \tau(1)\setminus\{\rho\}$, the definition of $w_\tau$ and $w_{\tau\setminus\rho}$ implies that
\[
u_{\rho'}* (w_\tau-w_{\tau\setminus\rho})=u_{\rho'}* w_\tau-u_{\rho'}* w_{\tau\setminus\rho}=z_{\rho'}-z_{\rho'}=0,
\]
so $w_\tau-w_{\tau\setminus\rho}$ is normal to $N_{\tau\setminus\rho,\R}$. This implies that $w_\tau-w_{\tau\setminus\rho}\in H_{\tau,\rho}^\circ$ if and only if $u_{\rho}*(w_\tau-w_{\tau\setminus\rho})>0$.
\item The definition of $w_\tau$ implies that
\[
u_{\rho}*(w_\tau-w_{\tau\setminus\rho})=u_{\rho}*w_\tau-u_{\rho}*w_{\tau\setminus\rho}=z_{\rho}-u_{\rho}*w_{\tau\setminus\rho}.
\]
\end{enumerate}

Now to prove the statement in the proposition regarding the cubical hypothesis, notice that $z$ is cubical with respect to $(\sigma,*)$ if and only if $w_\tau\in \tau^\circ$ for all $\tau\preceq\sigma$ (by definition), which holds if and only if $w_\tau\in H_{\tau,\rho}^\circ$ for all $\tau\preceq\sigma$ and $\rho\in\tau(1)$ (by the above half-space presentation), which holds if and only if $u_{\rho}* w_{\tau\setminus\rho}<z_\rho$ for all $\tau\preceq\sigma$ and $\rho\in\tau(1)$ (by the above argument). This last condition is equivalent to the one given in the proposition. To prove the statement in the proposition regarding the pseudo-cubical hypothesis, simply remove each $\circ$ and replace each $<$ and $>$ with $\leq$ and $\geq$ in the above arguments.
\end{proof}

As a consequence of the previous proposition, we obtain the following structural result concerning (pseudo-)cubical values.

\begin{proposition}\label{prop:cubecone}
The set of cubical values $\cC(\Sigma,*)\subseteq\R^{\Sigma(1)}$ is an open convex polyhedral cone, the set of pseudo-cubical values $\overline\cC(\Sigma,*)\subseteq\R^{\Sigma(1)}$ is a closed convex polyhedral cone, and $\cC(\Sigma,*)=\overline\cC(\Sigma,*)^\circ$.
\end{proposition} 

\begin{proof}
By elementary linear algebra considerations, it follows from the definition of the $w$-vectors that the coordinates of $w_\tau$ are homogeneous and linear in $z$. Thus, by Proposition~\ref{prop:charcub}, the set of cubical values are characterized by a finite set of strict inequalities that are homogeneous and linear in $z$, and the set of pseudo-cubical values are characterized by weakening the strict inequalities to allow for equality. The result then follows from standard results in polyhedral geometry.
\end{proof}

\subsection{Structure of normal complexes}

In this subsection, we prove various structural properties of normal complexes, including that the normal complex $C_{\Sigma,*}(z)$ is, in fact, a polytopal complex, and that the constituent polytopes are combinatorially equivalent to cubes when $z$ is cubical. We begin with the following description of the combinatorial structure of the polytopes $P_{\sigma,*}(z)$.

\begin{proposition}\label{prop:vertices}
Suppose that $z\in\overline\cC(\Sigma,*)$ and $\sigma\in\Sigma$.
\begin{enumerate}
\item The vertices of $P_{\sigma,*}(z)$ are 
\[
W=\{w_\tau\;|\;\tau\preceq\sigma\}. 
\]
\item For any pair of disjoint subsets $S_0,S_1\subseteq\sigma(1)$, there is a face $F_{S_0,S_1}\preceq P_{\sigma,*}(z)$ such that
\[
F_{S_0,S_1}\cap W=\{w_\tau\;|\;S_0\subseteq\sigma(1)\setminus\tau(1),\;S_1\subseteq\tau(1)\},
\]
and the faces of $P_{\sigma,*}(z)$ are 
\[
\{F_{S_0,S_1}\;|\;S_0,S_1\subseteq\sigma(1),\;S_0\cap S_1=\emptyset\}.
\]
\end{enumerate}
\end{proposition}

We note that, in general, both the vertex and face descriptions in Proposition~\ref{prop:vertices} are redundant. For example, if $w_\tau$ lies in a proper face $\tau'\prec\tau$, which can happen if $z$ is pseudo-cubical but not cubical, it then follows from the definition of the $w$-vectors that $w_{\tau'}=w_\tau$. As we'll see in Proposition~\ref{prop:cube}, there is no redundancy in these descriptions when $z$ is cubical.

\begin{proof}[Proof of Proposition~\ref{prop:vertices}]
Recall that
\[
P_{\sigma,*}(z)=\sigma\cap\big\{v\in N_\R\;|\;v* u_\rho\leq z_\rho\text{ for all }\rho\in\sigma{(1)}\big\}\subseteq N_{\sigma,\R}.
\]
It follows from this description that $P_{\sigma,*}(z)$ is an intersection of closed half-spaces in $N_{\sigma,\R}$---two half-spaces indexed by each $\rho\in\sigma(1)$---where the bounding hyperplanes are
\[
N_{\sigma\setminus\rho,\R}\;\;\;\text{ and }\;\;\;\{v\in N_{\sigma,\R}\;|\;v* u_\rho=z_\rho\}.
\]
Therefore, for each $\rho\in\sigma(1)$, we obtain two faces 
\[
F_{\rho}^0=P_{\sigma,*}(z)\cap N_{\sigma\setminus\rho,\R}\;\;\;\text{ and }\;\;\;F_{\rho}^1=P_{\sigma,*}(z)\cap\{v\in N_{\sigma,\R}\;|\;v* u_\rho=z_\rho\}.
\]
These faces may not be facets of $P_{\sigma,*}(z)$, but since they are obtained by intersecting with the hyperplanes associated to a half-space presentation of $P_{\sigma,*}(z)$, it follows that this set of faces contains all facets. In particular, this implies that every face of $P_{\sigma,*}(z)$ can be obtained as an intersection of some subset of the faces of the form $F_{\rho}^0$ and $F_{\rho}^1$.

We first prove the vertex description by induction on $\dim(\sigma)$. If $\dim(\sigma)=0$, then 
\[
P_{\sigma,*}(z)=\{0\}=\{w_\sigma\}, 
\]
proving the base case. Now suppose $\dim(\sigma)>0$. For each $\rho\in\sigma(1)$, we have
\begin{equation}\label{eq:vertices1}
F_{\rho}^0=P_{\sigma,*}(z)\cap N_{\sigma\setminus\rho,\R}=P_{\sigma\setminus\rho,*}(z)\cap\{v\in N_{\sigma\setminus\rho,\R}\;|\;v\cdot u_\rho\leq z_\rho\},
\end{equation}
where the second equality follows from the fact that $P_{\sigma,*}(z)$ has one more defining inequality than $P_{\sigma\setminus\rho,*}(z)$. By the induction hypothesis, the vertices of $P_{\sigma\setminus\rho,*}(z)$ are $\{w_\tau\;|\;\tau\preceq \sigma\setminus\rho\}$, and using the pseudo-cubical hypothesis and Proposition~\ref{prop:charcub}, we see that
\[
w_\tau*u_\rho\leq z_\rho\;\;\;\text{ for all }\;\;\;\tau\preceq\sigma\setminus\rho.
\]
Thus, $P_{\sigma\setminus\rho,*}(z)\subseteq\{v\in N_{\sigma\setminus\rho,\R}\;|\;v\cdot u_\rho\leq z_\rho\}$, and it follows from \eqref{eq:vertices1} that
\[
F_{\rho}^0=P_{\sigma\setminus\rho,*}(z).
\]
Therefore, the vertices of $P_{\sigma,*}(z)$ that lie in the face $F_{\rho}^0$ are equal to $\{w_\tau\;|\;\tau\preceq \sigma\setminus\rho\}$. Applying this same reasoning to all $\rho\in\sigma(1)$, it follows that the vertices of $P_{\sigma,*}(z)$ that are contained in at least one face of the form $F_{\rho}^0$ are equal to $\{w_\tau\;|\;\tau\prec\sigma\}$. It now remains to consider the vertices that do not lie in any of the faces of the form $F_{\rho}^0$. Noting that 
\[
\bigcap_{\rho\in\sigma(1)}F_{\rho}^1=\{w_\sigma\},
\]
it follows that there is at most one such vertex, and it is $w_\sigma$. Thus, we conclude that the vertices of $P_{\sigma,*}(z)$ are $W=\{w_\tau\;|\;\tau\preceq\sigma\}$.

Next, we justify the face description. Intersecting each face $F_\rho^0$ with the vertices, we claim that
\begin{equation}\label{eq:vertex1}
F_\rho^0\cap W=\{w_\tau\;|\;\rho\notin \tau(1)\}. 
\end{equation}
To prove this, first note that $w_\tau\in N_{\sigma\setminus\rho,\R}$ if $\rho\notin \tau(1)$ (by definition of $w_\tau$), which proves that $F_\rho^0\cap W\supseteq\{w_\tau\;|\;\rho\notin \tau(1)\}$. For the other inclusion, suppose that $w_\tau\in F_\rho^0\cap W$ and consider the cone $\tau'=\tau\cap(\sigma\setminus\rho)$. Since $w_\tau\in N_{\tau,\R}$ and $w_\tau\in F_\rho^0\subseteq N_{\sigma\setminus\rho,\R}$, it follows that $w_\tau\in N_{\tau',\R}\subseteq N_{\tau,\R}$. By definition of the $w$-vectors, this implies that
\[
w_\tau=w_{\tau'}\in\{w_\tau\;|\;\rho\notin \tau(1)\}.
\]
Similarly, intersecting each face $F_\rho^1$ with the vertices, we claim that
\begin{equation}\label{eq:vertex2}
F_\rho^1\cap W=\{w_\tau\;|\;\rho\in \tau(1)\}. 
\end{equation}
To prove this, first notice that $w_\tau*u_\rho=z_\rho$ if $\rho\in \tau(1)$ (by definition of $w_\tau$), which proves that $F_\rho^1\cap W\supseteq\{w_\tau\;|\;\rho\in \tau(1)\}$. To prove the other inclusion, suppose that $w_\tau\in F_\rho^1\cap W$ and consider the cone $\tau'\preceq\sigma$ with rays $\tau(1)\cup\{\rho\}$. Since $w_\tau*u_\eta=z_\eta$ for every $\eta\in\tau'(1)$, it follows from the definition of the $w$-vectors that
\[
w_\tau=w_{\tau'}\in\{w_\tau\;|\;\rho\in \tau(1)\}. 
\]

Now, for each pair of disjoint subsets $S_0,S_1\subseteq\sigma(1)$, define
\[
F_{S_0,S_1}=\bigcap_{\rho\in S_0}F_\rho^0\cap\bigcap_{\rho\in S_1}F_\rho^1.
\]
From \eqref{eq:vertex1} and \eqref{eq:vertex2}, we see that
\[
F_{S_0,S_1}\cap W=\{w_\tau\;|\;S_0\subseteq\sigma(1)\setminus\tau(1),\;S_1\subseteq\tau(1)\}.
\]
It remains to prove that every face of $P_{\sigma,*}(z)$ is of the form $F_{S_0,S_1}$ for some disjoint pair $S_0,S_1\subseteq\sigma(1)$, and we accomplish this by induction on $\dim(\sigma)$. If $\dim(\sigma)=0$, then the only face of $P_{\sigma,*}(z)$ is $F_{\emptyset,\emptyset}=P_{\sigma,*}(z)=\{0\}$. Suppose $\dim(\sigma)>0$ and let $F\preceq P_{\sigma,*}(z)$ be a face. Then $F$ is an intersection of faces of the form $F_{\rho}^0$ and $F_\rho^1$. If the intersection does not involve $F_{\rho}^0$ for any $\rho$, then there is nothing to prove. If the intersection involves $F_\rho^0$, then we can view $F$ as a face of $F_\rho^0=P_{\sigma\setminus\rho,*}(z)$. By induction, we have that
\[
F=\bigcap_{\eta\in S_0}F_\eta^0\cap\bigcap_{\eta\in S_1}F_\eta^1\preceq P_{\sigma\setminus\rho,*}(z)
\]
for some pair of disjoint subsets $S_0,S_1\subseteq\sigma(1)\setminus\{\rho\}$. As a face of $P_{\sigma,*}(z)$, we can then write
\[
F=\bigcap_{\eta\in S_0\cup\{\rho\}}F_\eta^0\cap\bigcap_{\eta\in S_1}F_\eta^1,
\]
which, upon observing that $S_0\cup\{\rho\}$ and $S_1$ are disjoint, completes the induction step, finishing the proof.
\end{proof}

As a first consequence of the combinatorial description of the polytopes $P_{\sigma,*}(z)$ given in Proposition~\ref{prop:vertices}, we have the following important result.

\begin{proposition}\label{prop:polycomplex}
If $z\in\overline\cC(\Sigma,*)$, then $C_{\Sigma,*}(z)$ is a polytopal complex.
\end{proposition}

\begin{proof}
Recall that
\[
C_{\Sigma,*}(z)=\bigcup_{\sigma\in\Sigma}\widehat P_{\sigma,*}(z)
\]
where $\widehat P_{\sigma,*}(z)$ is the polytopal complex consisting of $P_{\sigma,*}(z)$ and its faces. In order to prove that a collection of polytopes and their faces form a polytopal complex, it suffices to check that the polytopes meet along common faces. Consider two polytopes $P_{\sigma_1,*}(z)$ and $P_{\sigma_2,*}(z)$ associated to cones $\sigma_1,\sigma_2\in\Sigma$. Let $\tau=\sigma_1\cap\sigma_2\in\Sigma$, and notice that $P_{\sigma_1,*}(z)\cap P_{\sigma_2,*}(z)\subseteq N_{\tau,\R}.$ For $i=1,2$, define $S_0^i=\sigma_i(1)\setminus\tau(1)\subseteq\sigma_i(1)$ and notice that 
\[
P_{\sigma_i,*}(z)\cap N_{\tau,\R} = F_{S_i,\emptyset}\preceq P_{\sigma_i,*}(z).
\]
By Proposition~\ref{prop:vertices},
\[
F_{S_i,\emptyset}=\conv(w_\pi\;|\;\pi\preceq\tau), 
\]
from which it follows that 
\[
P_{\sigma_1,*}(z)\cap P_{\sigma_2,*}(z)=\conv(w_\pi\;|\;\pi\preceq\tau)\preceq P_{\sigma_i}(z,*),
\]
showing that the intersection is a face of both $P_{\sigma_1,*}(z)$ and $P_{\sigma_2,*}(z)$.
\end{proof}

As mentioned above, the combinatorial description of $P_{\sigma,*}(z)$ in Proposition~\ref{prop:vertices} may be highly redundant; however, if we restrict to the cubical setting, that redundancy goes away. The next result proves this, while also giving a justification for the term ``cubical.''

\begin{proposition}\label{prop:cube}
If $z\in\cC(\Sigma,*)$ is cubical and $\sigma\in\Sigma(k)$, then the polytope $P_{\sigma,*}(z)$ is combinatorially equivalent to a $k$-cube.
\end{proposition}

\begin{proof}
We must show that the face lattice of $P_{\sigma,*}(z)$ is isomorphic to the face lattice of the unit cube $[0,1]^k\subseteq\R^k$. Notice that the faces $[0,1]^k$ are of the form
\[
\widetilde F_{S_0,S_1}=[0,1]^k\cap\{x_\rho=0\;|\;\rho\in S_0\}\cap\{x_\rho=1\;|\;\rho\in S_1\}
\]
where $S_0,S_1\subseteq\{1,\dots,k\}$ are disjoint subsets. In particular, there are $2^k$ vertices given by
\[
\widetilde W=\Big\{\widetilde w_\tau=\sum_{\rho\in \tau}e_\rho\;|\;\tau\subseteq \{1,\dots,k\}\Big\},
\]
and the face lattice of $[0,1]^k$ is determined by the vertex-face containments:
\[
\widetilde W\cap \widetilde F_{S_0,S_1}=\{\widetilde w_\tau\;|\;S_0\subseteq\{1,\dots,k\}\setminus\tau,\;S_1\subseteq\tau\}.
\]

Comparing the above description with the combinatorial description of $P_{\sigma,*}(z)$ given in Proposition~\ref{prop:vertices}, we see that the two descriptions are equivalent as long as the vertices $w_\tau\in P_{\sigma,*}(z)$ are all distinct. By the cubical hypothesis, we know that $w_\tau\in\tau^\circ$ for all $\tau\preceq\sigma$. Along with the observation that $\tau_1^\circ\cap\tau_2^\circ=\emptyset$ for all $\tau_1\neq\tau_2$, we conclude that $w_{\tau_1}\neq w_{\tau_2}$ for all $\tau_1\neq\tau_2$, completing the proof.
\end{proof}

\subsection{Triangulating normal complexes}

Our next aim is to construct a triangulation of the normal complex $C_{\Sigma,*}(z)$ for all pseudo-cubical values $z\in\overline\cC(\Sigma,*)$. To describe the triangulation, we first require some additional notation. Let $\sigma\in\Sigma(k)$, define $L(\sigma)$ to be the set of labeling bijections $f:\{1,\dots,k\}\rightarrow \sigma(1)$. For each $f\in L(\sigma)$ and $0\leq j\leq k$, let $\sigma(f,j)\preceq\sigma$ be the face of $\sigma$ with rays indexed by $\{f(i)\;|\;i\leq j\}$. Define polytopes
\[
\Delta(\sigma,f)=\conv(w_{\sigma(f,0)},\dots,w_{\sigma(f,k)}),
\]
where we note that the first vector is just the origins: $w_{\sigma(f,0)}=0$. In the next example, we depict how these polytopes fit together in the generic cubical setting.

\begin{example}
Let $N_\R=\R^3$, let $u_1,u_2,u_3$ be the standard basis vectors and let $\sigma$ be the first octant. Let $*$ be the dot product and set $z_1=z_2=z_3=1$. Then $P_{\sigma,*}(z)$ is the unit cube in $\R^3$ and for any face $\tau\preceq\sigma$, we have $w_\tau=\sum_{\rho_i\in \tau(1)}u_i$. Each labeling function $f\in L(\sigma)$ determines a simplex, and these simplices (along with their faces) triangulate the unit cube as depicted below. Note that the origin is in the lower left-hand corner of this image and the vector $(1,1,1)$ is in the upper right-hand corner.

\begin{center}
\begin{tikzpicture}
\node at (0,0) {\includegraphics[width=.25\textwidth]{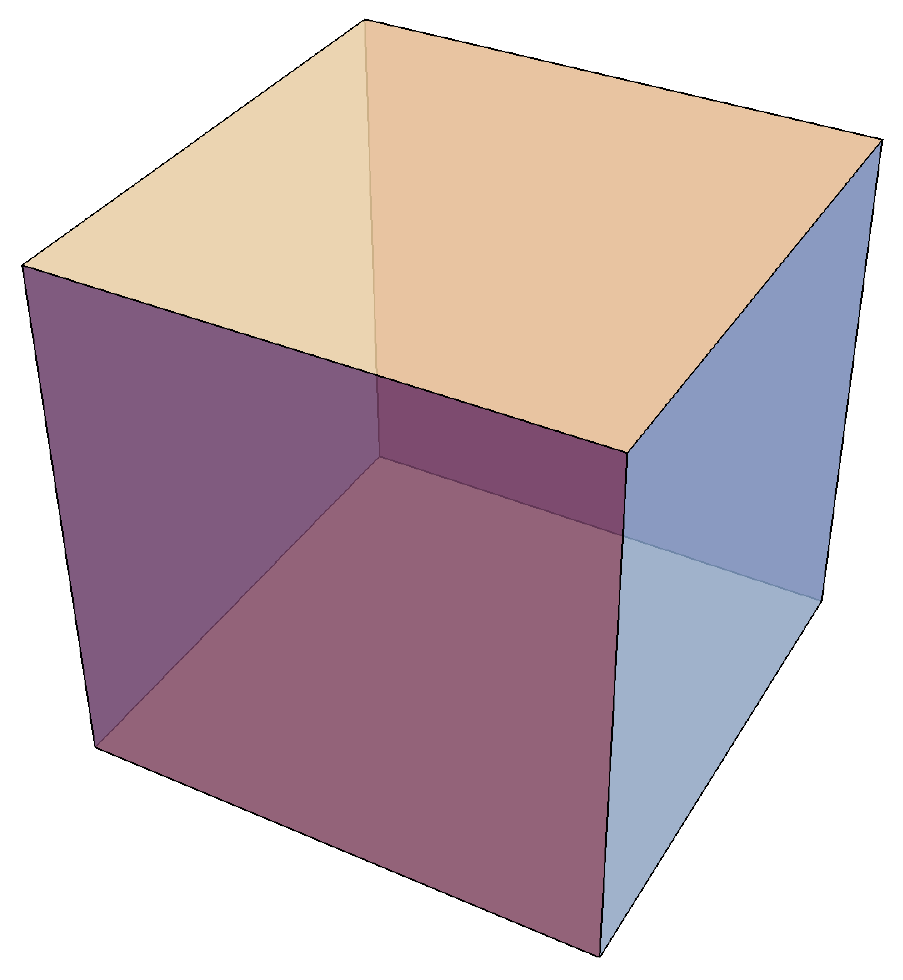}};
\draw[->] (3,0) -- (4,0);
\node at (7,0) {\includegraphics[width=.3\textwidth]{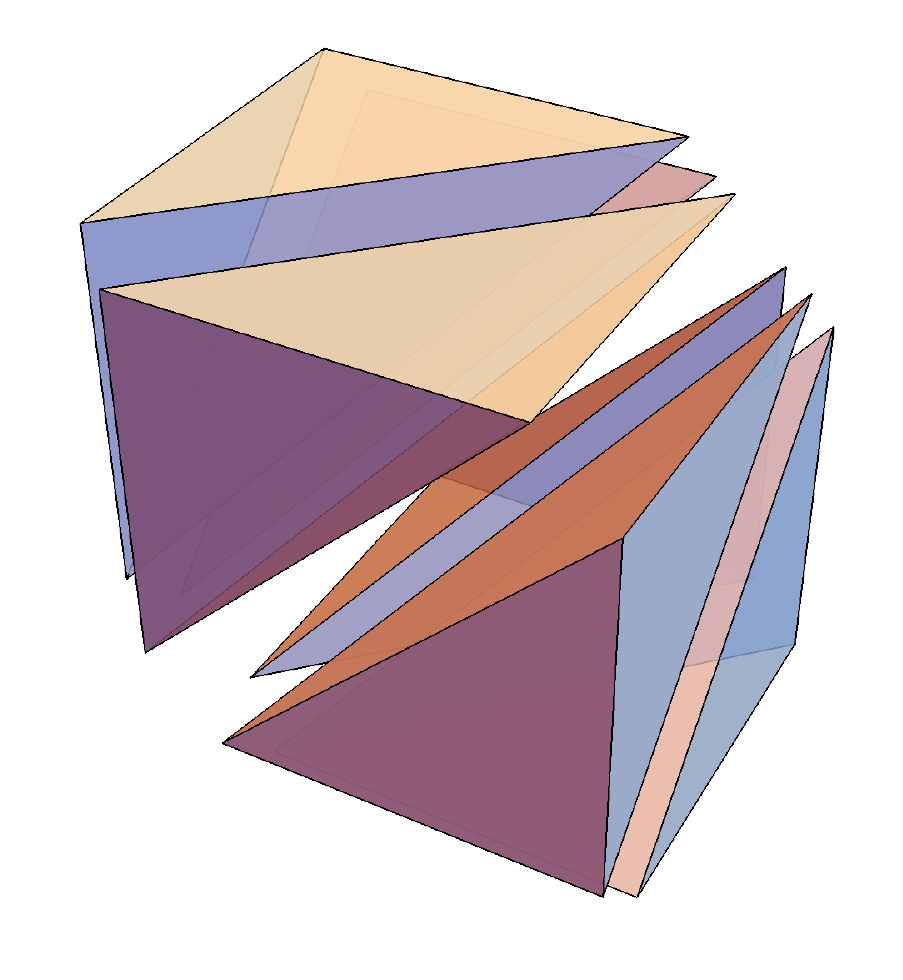}};
\end{tikzpicture}
\end{center}
\end{example}

While the previous example illustrated how the polytopes $\Delta(\sigma,f)$ fit together in the generic cubical setting, the situation can become much more degenerate in the pseudo-cubical setting, when some of the vertices of $P_{\sigma,*}(z)$ are allowed to coincide. We give one particularly degenerate illustration in the next example.

\begin{example}
Let $N_\R=\R^3$, let $u_1=(1,0,0)$, $u_2=(1,1,0)$, and $u_3=(1,1,1)$, and set $\sigma=\cone(u_1,u_2,u_3)$. Let $*$ be the dot product and set $z_1=1$, $z_2=2$, and $z_3=3$. Then $P_{\sigma,*}(z)=\conv(0,u_1,u_2,u_3)$, which we have depicted below.

\begin{center}
\includegraphics[width=.25\textwidth]{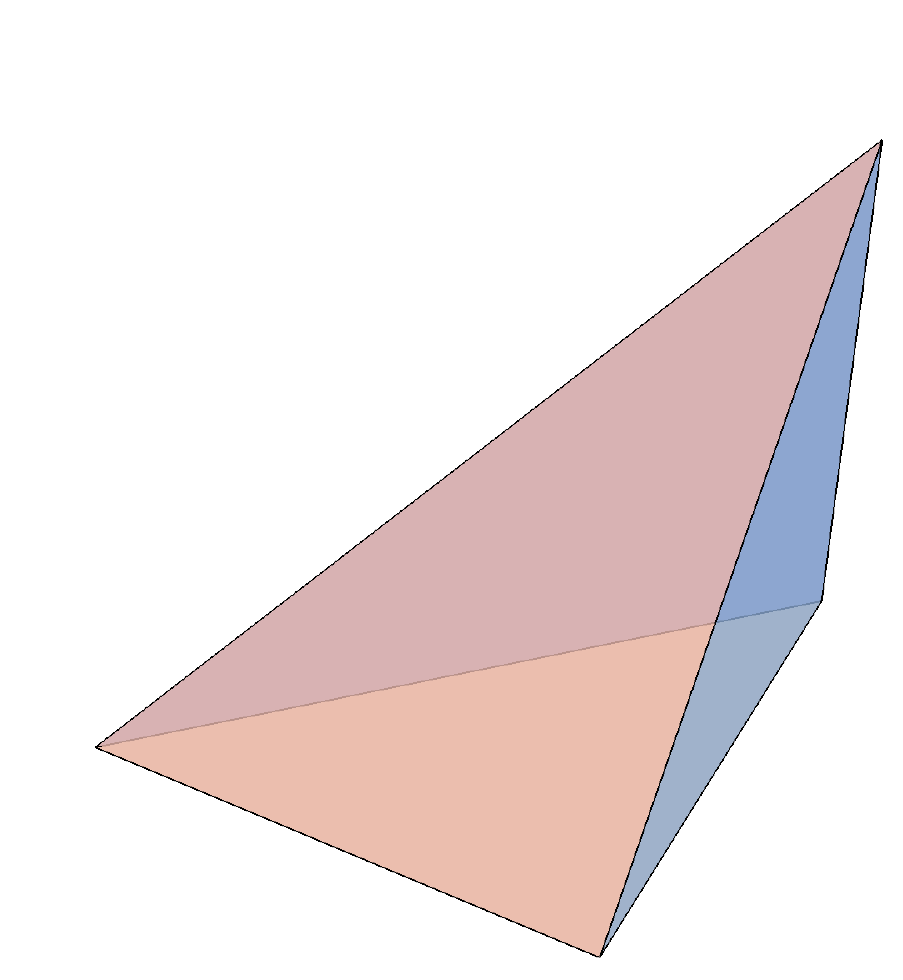}
\end{center}
It can be checked that, for any face $\tau\preceq\sigma$, we have
\[
w_\tau=\begin{cases}
u_1 & \tau(1)=\{\rho_1\},\\
u_2 & \rho_2\in\tau(1)\text{ and }\rho_3\notin\tau(1),\\
u_3 & \rho_3\in\tau(1).
\end{cases}
\]
It follows that $\Delta(\sigma,f)=P_{\sigma,*}(z)$ if $f(i)=\rho_i$ for each $i$, and for every other labeling function, $\Delta(\sigma,f)$ is a proper face of $P_{\sigma,*}(z)$. Even though there is a lot of redundancy in this pseudo-cubical setting, the simplices $\Delta(\sigma,f)$ (along with their faces) still triangulate $P_{\sigma,*}(z)$.
\end{example}

In each of the previous two examples, we saw that the simplices $\Delta(\sigma,f)$ and their faces triangulate the polytope $P_{\sigma,*}(z)$---we now aim to prove this in general. For each $\sigma\in\Sigma$ and $f\in L(\sigma)$, let $\widehat\Delta(\sigma,f)$ denote the polytopal complex consisting of $\Delta(\sigma,f)$ and its faces. The next result will be key to computing volumes of normal complexes in the next section.

\begin{proposition}\label{prop:triangulate}
For any $z\in\overline\cC(\Sigma,*)$, the collection
\[
\bigcup_{\sigma\in\Sigma(d)\atop f\in L(\sigma)}\widehat\Delta(\sigma,f)
\]
is a triangulation of $C_{\Sigma,*}(z)$. Furthermore, the function
\[
f\in L(\sigma)\mapsto\Delta(\sigma,f)\in\{\text{polytopes in }N_{\sigma,\R}\}
\]
is injective when restricted to the preimage of $d$-dimensional polytopes.
\end{proposition}

\begin{proof}
Fix $\sigma\in\Sigma(d)$ and consider the polytope $P_{\sigma,*}(z)$. We prove that the collection
\begin{equation}\label{eq:triangulation}
\bigcup_{f\in L(\sigma)}\widehat\Delta(\sigma,f)
\end{equation}
is a triangulation of $\widehat P_{\sigma,*}(z)$. To do so, we argue that \eqref{eq:triangulation} results from a sequence of \emph{pulling subdivisions}---a procedure that we now recall. 

If $C$ is a polytopal complex and $v\in C$ is a vertex, then the \textbf{pulling subdivision of $C$ at $v$}, denoted $\mathrm{pull}_vC$, is the polytopal complex obtained by replacing every polytope $P\in C$ that contains $v$ with the collection of pyramids $\conv(v,F)$ for all proper faces $F\prec P$. For example, the pulling subdivision of a polygon (and its faces) at a vertex is the triangulation obtained by connecting that vertex to every other vertex of the polygon. A few important properties of pulling subdivisions that can be readily checked from the definition are:
\begin{enumerate}
\item The polytopal complexes $C$ and $\mathrm{pull}_vC$ have the same support;
\item The vertex $v$ is an \emph{apex} of $\mathrm{pull}_vC$, meaning that $v$ is adjacent to every vertex in every polytope of $\mathrm{pull}_vC$ that contains $v$;
\item If $v$ is an apex of $C$, then $\mathrm{pull}_vC=C$;
\item If $w$ is an apex of $C$, then it is also an apex of $\mathrm{pull}_vC$.
\end{enumerate}
It follows from these properties that sequentially performing a pulling subdivisions at every vertex of a polytopal complex results in a polytopal complex for which every vertex is an apex; in other words, it results in a triangulation.

We now claim that \eqref{eq:triangulation} is obtained by an iterated sequence of pulling subdivisions of $\widehat P_{\sigma,*}(z)$, where we first subdivide at the vertex $w_{\{0\}}=0$, then at the vertices $\{w_\tau\;|\;\tau\in\sigma(1)\}$ (in any order), then at the vertices $\{w_\tau\;|\;\tau\in\sigma(2)\}$ (in any order), and so forth. To prove this, let $\widehat P_{\sigma,*}(z)_k$ denote the polytopal complex resulting from the first $k$ steps of this process; we claim that
\begin{equation}\label{eq:trialg}
\widehat P_{\sigma,*}(z)_k=\bigcup_{\pi\in\sigma(k)\atop f\in L(\pi)}\widehat P(\pi,f)
\end{equation}
where
\[
P(\pi,f)=\conv(w_{\pi(f,0)},\dots,w_{\pi(f,k-1)},F_{\emptyset,\pi(1)}),
\]
and the faces $F_{\emptyset,\pi(1)}\preceq P_{\sigma,*}(z)$ are those described in Proposition~\ref{prop:vertices}. Upon observing that $P(\sigma,f)=\Delta(\sigma,f)$, we see that the triangulation \eqref{eq:triangulation} is the $k=d$ case of \eqref{eq:trialg}.

We prove \eqref{eq:trialg} by induction on $k$. For the base case $k=0$, it suffices to notice that 
\[
P_{\sigma,*}(z)=F_{\emptyset,\emptyset}=\conv(F_{\emptyset,\emptyset}).
\]
To prove the induction step, assume that \eqref{eq:trialg} holds for some $k$. By definition, $\widehat P_{\sigma,*}(z)_{k+1}$ is the pulling subdivision of $\widehat P_{\sigma,*}(z)_{k}$ at $\{w_\tau\;|\;\tau\in\sigma(k)\}$, so using the induction hypothesis, we can compute this in terms of the right-hand side of \eqref{eq:trialg}. Fix $\tau\in\sigma(k)$. To compute $\mathrm{pull}_{w_\tau}\widehat P_{\sigma,*}(z)_k$, we first identify which polytopes $P(\pi,f)$ in the right-hand side of \eqref{eq:trialg} contain $w_\tau$ as a vertex. There are two possiblilities: either $w_\tau=w_{\pi(f,i)}$ for some $\pi$, $f$, and $i$, or $\pi=\tau$, in which case $w_\tau\in F_{\emptyset,\pi(1)}$. In the first case, $w_\tau$ is already an apex of $\widehat P(\pi,f)$, so $\mathrm{pull}_{w_\tau}\widehat P(\pi,f)=\widehat P(\pi,f)$. Thus, it remains to compute the pulling subdivision in the second case:  $\mathrm{pull}_{w_\tau}\widehat P(\tau,f)$.

To compute $\mathrm{pull}_{w_\tau}\widehat P(\tau,f)$, notice that $w_\tau\in F_{\emptyset,\tau(1)}$ and, by Proposition~\ref{prop:vertices}, every face of $F_{\emptyset,\tau(1)}$ that does not contain $w_\tau$ is contained in some face of the form $F_{\emptyset,\tau(1)\cup\{\rho\}}$ for some $\rho\in\sigma(1)\setminus\tau(1)$. It follows that every face of $P(\tau,f)$ that does not contain $w_\tau$ is contained in some face of the form
\[
\conv(w_{\tau(f,0)},\dots,w_{\tau(f,k-1)},F_{\emptyset,\tau(1)\cup\{\rho\}}).
\]
Noting that $\tau=\tau(f,k)$, it then follows from the definition of the pulling subdivision that
\[
\mathrm{pull}_{w_\tau}\widehat P(\tau,f)=\bigcup_{\rho\in\sigma(1)\setminus\tau(1)}\widehat \conv(w_{\tau(f,0)},\dots,w_{\tau(f,k-1)},w_{\tau(f,k)},F_{\emptyset,\tau(1)\cup\{\rho\}}).
\]
Varying over all $f\in L(\tau)$ and $\tau\in\sigma(k)$, it then follows that
\[
\widehat P_{\sigma,*}(z)_{k+1}=\bigcup_{\pi\in\sigma(k+1)\atop f\in L(\pi)}\widehat P(\pi,f),
\]
completing the induction step.

To prove the final statement in the proposition, the key observation we require is that the face $F_{\emptyset,S}$ has dimension at most $d-|S|$, and whenever $\dim(F_{\emptyset,S})=d-|S|$, we have
\begin{equation}\label{eq:facecontain}
F_{\emptyset,S}\subseteq F_{\emptyset,\{\rho\}}=F_\rho^1\;\;\;\text{ if and only if }\;\;\;\rho\in S.
\end{equation}
In other words, if $\dim(F_{\emptyset,S})=d-|S|$, then $F_{\emptyset,S}$ uniquely determines the set of rays in $S$. Property \eqref{eq:facecontain} follows from the fact---discussed in the proof of Proposition~\ref{prop:vertices}---that the facets of $P_{\sigma,*}(z)$ are a subset of the faces of the form $F_{\rho}^0$ and $F_{\rho}^1$. 

Now suppose that $\Delta(\sigma,f)$ is $d$-dimensional; we must prove that $f$ is uniquely determined by $\Delta(\sigma,f)$. Using that $\{w_{\sigma(f,1)},\dots,w_{\sigma(f,d)}\}$ are linearly independent and contained in $F_{\emptyset,\sigma(f,1)}$, we see that $\dim(F_{\emptyset,\sigma(f,1)})=d-1$. By \eqref{eq:facecontain}, this implies that the face spanned by the nonzero vectors in $\Delta(\sigma,f)$ uniquely determines $\sigma(f,1)$, and thus determines $f(1)$. Next, suppose we have used $\Delta(\sigma,f)$ to uniquely determine $f(1),\dots,f(k-1)$; we must show that we can then uniquely determine $f(k)$. Using that $\{w_{\sigma(f,k)},\dots,w_{\sigma(f,d)}\}$ are linearly independent and contained in $F_{\emptyset,\sigma(f,k)(1)}$, we see that $\dim(F_{\emptyset,\sigma(f,k)(1)})=d-k$. By \eqref{eq:facecontain}, this implies that the face spanned by $\{w_{\sigma(f,k)},\dots,w_{\sigma(f,d)}\}$ uniquely determines $\sigma(f,k)$, thereby determining $f(k)$. This completes the induction step, finishing the proof.
\end{proof}

\section{Volume computations}\label{sec:computingvolumes}

The main result of this section is the derivation of an explicit formula for weighted volumes of normal complexes. Throughout this section, let $\Sigma\subseteq N_\R$ denote a simplicial $d$-fan, and let $*\in\Inn(N_\R)$ be an inner product.

\subsection{Normalizing volume}

In this subsection, we discuss a preparatory result that allows us to compute normalized volumes using determinants. Let $\sigma\in\Sigma{(k)}$ and consider the vector space $N_{\sigma,\R}$. Notice that any volume function $\big\{\text{polytopes in }N_{\sigma,\R}\big\}\rightarrow\R_{\geq 0}$ is uniquely determined by its value on the simplex $\Delta_\sigma=\conv(\{0\}\cup\{u_\rho\;|\;\rho\in\sigma(1)\})$, and any two volume functions differ by a scalar multiple. By restricting the inner product $*\in\Inn(N_\R)$, we obtain an inner product $*\in\Inn(N_{\sigma,\R})$, and this inner product allows us to define the \emph{Euclidean (simplicial) volume function}
\[
\vol_{\sigma}:\big\{\text{polytopes in }N_{\sigma,\R}\big\}\rightarrow\R_{\geq 0},
\]
which is normalized so that the simplex associated to any orthonormal basis (with respect to $*$) has unit volume. A linear algebra exercise shows that the Euclidean volume of the simplex $\Delta_\sigma$ is given by the formula
\begin{equation}\label{eq:simplexeuclideanvolume}
\vol_{\sigma}(\Delta_\sigma)=\sqrt{\det(G_{\sigma})}
\end{equation}
where $G_\sigma$ is the Gram matrix
\[
G_\sigma=(u_{\rho}* u_\eta)_{\rho,\eta\in\sigma(1)}.
\]
In regards to computing volumes of normal complexes, we require the volume function 
\[
\Vol_{\sigma}:\big\{\text{polytopes in }N_{\sigma,\R}\big\}\rightarrow\R_{\geq 0},
\]
which is normalized so that the volume of a fundamental simplex of the lattice $M_\sigma\subseteq N_{\sigma,\R}$ has unit volume. The next result allows us to compute the volume function $\Vol_{\sigma}(-)$ in terms of the Euclidean volume function $\vol_{\sigma}(-)$.

\begin{lemma}\label{lemma.normalizedvolume}
For any polytope $P\subseteq N_{\sigma,\R}$, we have
\[
\Vol_{\sigma}(P)=\sqrt{\det(G_{\sigma})}\vol_{\sigma}(P).
\]
\end{lemma}

\begin{proof}
Let $\{v_\eta\;|\;\eta\in\sigma\}\subseteq M_\sigma$ be the dual basis of $\{u_\rho\;|\;\rho\in\sigma(1)\}\subseteq N_\sigma$ determined by the property that
\[
u_\rho * v_\eta =\delta_{\rho,\eta},
\]
where $\delta_{\rho,\eta}$ is the Kronecker delta function. Using the identification of $M_{\sigma,\R}$ and $N_{\sigma,\R}$ given by the inner product $*$, write the change of basis transformation as $T_\sigma=(a_{\rho,\eta})_{\rho,\eta\in\sigma(1)}$ where
\[
u_\rho=\sum_{\eta\in\sigma(1)} a_{\rho,\eta} v_\eta.
\]
Notice that
\[
G_\sigma=(u_\rho * u_\eta)_{\rho,\eta\in\sigma(1)}=(a_{\rho,\eta})_{\rho,\eta\in\sigma(1)}=T_\sigma,
\]
where the second equality follows from replacing each $u_\rho$ with its expression in terms of $v_\eta$. Define 
\[
\Delta^\sigma=\conv(\{0\}\cup\{v_\eta\;|\;\eta\in\sigma\}),
\]
so that
\[
T_\sigma(\Delta^\sigma)=\Delta_{\sigma}.
\]
Using that linear transformations rescale volumes by the absolute value of their determinant, we see that
\begin{align*}
\Vol_{\sigma}(\Delta_\sigma)&=|\det(T_\sigma)|\Vol_{\sigma}(\Delta^{\sigma})\\
&=\det(G_\sigma)\Vol_{\sigma}(\Delta^{\sigma})\\
&=\det(G_\sigma).
\end{align*}
The second equality follows from the above observation that $G_\sigma=T_\sigma$ and the fact that Gram determinants are always nonnegative, and the third equality follows from the fact that $\Delta^{\sigma}$ is a unit simplex in $M_\sigma$. Therefore, combining these computations with \eqref{eq:simplexeuclideanvolume}, we see that
\[
\Vol_{\sigma}(\Delta_\sigma)=\sqrt{\det(G_{\sigma})}\vol_{\sigma}(\Delta_\sigma),
\]
proving that the scaling factor between these volume functions is $\sqrt{\det(G_{\sigma})}$.
\end{proof}

\subsection{Volumes of normal complexes}

We now present a formula for volumes of normal complexes. Recall that, for any pseudo-cubical value $z\in\overline\cC(\Sigma,*)$, the volume of the normal complex $C_{\Sigma,*}(z)$ weighted by $\omega$ is defined by
\[
\Vol(C_{\Sigma,*}(z);\omega)=\sum_{\sigma\in\Sigma{(d)}}\omega(\sigma)\Vol_{\sigma}(P_{\sigma,*}(z)).
\]
The next result computes an explicit formula for these volumes.

\begin{theorem}\label{thm:volumecomputation}
For any $z\in\overline\cC(\Sigma,*)$ and $\sigma\in\Sigma(k)$ we have
\[
\Vol_{\sigma}(P_{\sigma,*}(z))=\det(G_\sigma)\sum_{f\in L(\sigma)}\prod_{j=1}^k\frac{\det(G_{\sigma(f,j),f(j)})}{\det(G_{\sigma(f,j)})}
\]
where the notation is defined as follows:
\begin{itemize}
\item for $\sigma\in\Sigma(k)$, the set $L(\sigma)$ is the set of bijections $f:\{1,\dots,k\}\rightarrow\sigma(1)$;
\item for $f\in L(\sigma)$ and $1\leq j\leq k$, the cone $\sigma(f,j)\preceq\sigma$ has rays indexed by $\{{f(i)}\;|\;i\leq j\}$;
\item the matrix $G_\sigma$ is defined by $G_\sigma=(u_\rho*u_\eta)_{\rho,\eta\in\sigma(1)}$;
\item the matrix $G_{\sigma,\rho}(z)$ is obtained by replacing the $\rho$th column of $G_\sigma$ with $z_\sigma=(z_\eta)_{\eta\in\sigma(1)}$.
\end{itemize}
\end{theorem}

\begin{proof}
By Proposition~\ref{prop:triangulate}, we can write
\[
\Vol_{\sigma}(P_{\sigma,*}(z))=\sum_{f\in L(\sigma)}\Vol_{\sigma}(\Delta(\sigma,f)),
\]
where
\[
\Delta(\sigma,f)=\conv\big(w_{\sigma(f,0)},\dots,w_{\sigma(f,k)}\big).
\]
Therefore, it suffices to prove that 
\[
\Vol_{\sigma}(\Delta(\sigma,f))=\det(G_\sigma)\prod_{j=1}^k\frac{\det(G_{\sigma(f,j),f(j)})}{\det(G_{\sigma(f,j)})}.
\]

It follows from Lemma~\ref{lemma.normalizedvolume} that
\[
\Vol_{\sigma}(\Delta(\sigma,f))=\sqrt{\det(G_\sigma)}\vol_{\sigma}(\Delta(\sigma,f)).
\]
In order to compute the Euclidean volume of $\Delta(\sigma,f)$, we start by writing each $w_\tau$ as a linear combination of the vectors $\{u_\rho\;|\;\rho\in\tau(1)\}$:
\[
w_\tau=\sum_{\rho\in\tau(1)}a_{\tau,\rho} u_\rho.
\]
Define the matrix
\[
T_f=\big(a_{\sigma(f,j),\rho} \big)_{1\leq j\leq k\atop \rho\in\sigma(1)},
\] 
so that $T_f(\Delta_\sigma)=\Delta(\sigma,f)$, where $\Delta_\sigma=\conv(\{0\}\cup\{u_\rho\;|\;\rho\in\sigma(1)\})$. Then
\begin{align*}
\vol_{\sigma}(\Delta(\sigma,f))&=|\det(T_f)|\vol_{\sigma}(\Delta_\sigma)\\
&=|\deg(T_f)|\sqrt{\det(G_\sigma)},
\end{align*}
where the second equality follows from \eqref{eq:simplexeuclideanvolume}. By definition, notice that $a_{\sigma(f,j),\rho}=0$ for all $\rho\notin\sigma(f,j)(1)$. It follows that, up to a sign, $\det(T_f)$ is the product of the entries 
\[
\{a_{\sigma(f,j),f(j)}\;|\;j=1,\dots,k\}. 
\]
All of these entries are nonnegative by the pseudo-cubical assumption, so
\[
|\det(T_f)|=|a_{\sigma(f,1),f(1)}\cdots a_{\sigma(f,k),f(k)}|=a_{\sigma(f,1),f(1)}\cdots a_{\sigma(f,k),f(k)}.
\]
Combining the above observations, we have proved that
\[
\Vol_{\sigma}(\Delta(\sigma,f))=\det(G_\sigma)\prod_{j=1}^ka_{\sigma(f,j),f(j)}.
\]
It remains to compute each $a_{\sigma(f,j),f(j)}$.

Recall that $w_{\sigma(f,j)}$ is defined by $j$ linear equations
\[
w_{\sigma(f,j)}*u_\rho=z_\rho\;\;\;\text{ with }\;\;\;\rho\in \sigma(f,j)(1)
\]
Writing
\[
w_{\sigma(f,j)}=\sum_{\rho\in\sigma(f,j)(1)} a_{\sigma(f,j),\rho} u_{\rho},
\]
these linear equations can be encoded in a matrix equation
\[
G_{\sigma(f,j)}\cdot (a_{\sigma(f,j),\rho})_{\rho\in\sigma(f,j)(1)}=(z_\rho)_{\rho\in\sigma(f,j)(1)}.
\]
By Cramer's rule, it then follows that
\begin{equation}\label{eq:positivity0}
a_{\sigma(f,j),\rho}=\frac{\det(G_{\sigma(f,j),\rho}(z))}{\det(G_{\sigma(f,j)})},
\end{equation}
and we conclude that
\[
\Vol_{\sigma}(\Delta(\sigma,f))=\det(G_\sigma)\prod_{j=1}^k\frac{\det(G_{\sigma(f,j),f(j)}(z))}{\det(G_{\sigma(f,j)})},
\]
completing the proof.
\end{proof}

\section{Square-free expressions}

In this section, we derive a formula for products of divisors in Chow rings of simplicial fans as linear combinations of monomials that are square free in the generators and whose coefficients are closely related to the volume computations of the previous section. 

Let $\Sigma\subseteq N_\R$ denote a simplicial $d$-fan, and for any $z\in\R^{\Sigma(1)}$, define
\[
D(z)=\sum_{\rho\in\Sigma(1)}z_\rho X_\rho\in A^1(\Sigma).
\]
Our main result of this section is the following.\footnote{Theorem~\ref{thm:squarefree} as written here is stronger than a result that was written in a preliminary draft of this paper, and the authors thank Chris Eur for an enlightening conversation that led to this strengthening.}

\begin{theorem}\label{thm:squarefree}
For any inner product $*\in\Inn(N_\R)$ and values $z_1,\dots,z_k\in\R^{\Sigma(1)}$, we have
\[
D(z_1)\cdots D(z_k)=\sum_{\sigma\in\Sigma(k)}\Big(\det(G_\sigma)\sum_{f\in L(\sigma)}\prod_{j=1}^k\frac{\det(G_{\sigma(f,j),f(j)}(z_j))}{\det(G_{\sigma(f,j)})}\Big)X_\sigma\in A^\bullet(\Sigma),
\]
where all notation is as in Theorem~\ref{thm:volumecomputation}. Moreover, the coefficients are nonnegative if $z_1,\dots,z_k\in\overline\cC(\Sigma,*)$ and positive if $z_1,\dots,z_k\in\cC(\Sigma,*)$. 
\end{theorem}

\begin{proof}
We prove the formula by induction on $k$, the base case being $k=0$, in which both sides of the equation are $1$. Assume that $k\geq 1$ and that the result is valid for $k-1$. Then
\begin{equation}\label{eq:sqfreeindhyp}
D(z_1)\cdots D(z_k)=\sum_{\tau\in\Sigma(k-1)}\Big(\det(G_\tau)\sum_{g\in L(\tau)}\prod_{j=1}^{k-1}\frac{\det(G_{\tau(g,j),g(j)}(z_j))}{\det(G_{\tau(g,j)})}\Big)X_\tau\sum_{\rho\in\Sigma(1)}z_{k,\rho}X_\rho.
\end{equation}
Given $\tau\in \Sigma(k-1)$ and $g\in L(\tau)$, the definition of the ideal $\J$ leads to the following system of $k-1$ linear equations:
\[
\sum_{i=1}^{k-1}u_{g(j)}*u_{g(i)}X_{g(i)}=-\sum_{\rho\notin\tau(1)}u_{g(j)}*u_\rho X_\rho.
\]
Solving this system using Cramer's rule, we have, for all $i=1,\dots,k-1$,
\[
X_{g(i)}=\frac{-1}{\det(G_\tau)}\sum_{\rho\notin\tau(1)}\det(G_{\tau}\stackrel{i}{\leftarrow}(u_{g(j)}*u_\rho)_j)X_\rho,
\]
where the rows and columns in $G_\tau$ are ordered by the labeling function $g\in L(\tau)$, and $G_{\tau}\stackrel{i}{\leftarrow}(u_{g(j)}*u_\rho)_j$ is the matrix obtained from $G_\tau$ by replacing the $i$th column with $(u_{g(j)}*u_\rho)_j$. It then follows that
\begin{align}\label{eq:sqfreework}
\nonumber X_\tau \sum_{\rho\in\Sigma(1)}z_{k,\rho} X_\rho&=X_\tau\sum_{\rho\notin\tau}X_\rho\Big(z_{k,\rho}-\sum_{i=1}^{k-1}\frac{z_{k,g(i)}\det(G_{\tau}\stackrel{i}{\leftarrow}(u_{g(j)}*u_\rho)_j)}{\det(G_\tau)}\Big)\\
\nonumber&=\sum_{\sigma\in\Sigma(k)\atop \tau\prec\sigma}X_\sigma\Big(z_{k,\rho}-\sum_{i=1}^{k-1}\frac{z_{k,g(i)}\det(G_{\tau}\stackrel{i}{\leftarrow}(u_{g(j)}*u_{\sigma\setminus\tau})_j)}{\det(G_\tau)}\Big)\\
&=\sum_{\sigma\in\Sigma(k)\atop \tau\prec\sigma}X_\sigma\frac{\det(G_{\sigma,\sigma\setminus\tau}(z_k))}{\det(G_\tau)},
\end{align}
where the final equality follows from expanding the numerator in the final expression along the last column. Defining $f\in L(\sigma)$ from $g\in L(\tau)$ by
\[
f(j)=\begin{cases}
g(j) &\text{if }j>k,\\
\sigma\setminus\tau &\text{if }j=k,
\end{cases}
\]
and substituting \eqref{eq:sqfreework} into \eqref{eq:sqfreeindhyp}, we then conclude that
\[
D(z_1)\cdots D(z_k)=\sum_{\sigma\in\Sigma(k)}\Big(\det(G_\sigma)\sum_{f\in L(\sigma)}\prod_{j=1}^k\frac{\det(G_{\sigma(f,j),f(j)}(z_j))}{\det(G_{\sigma(f,j)})}\Big)X_\sigma,
\]
completing the induction step.

To prove the positivity statements, it is enough to argue that
\begin{equation}\label{eq:positivity}
\frac{\det(G_{\sigma,\rho}(z))}{\det(G_{\sigma})}
\end{equation}
is nonnegative when $z\in\overline\cC(\Sigma,*)$ and positive when $z\in\cC(\Sigma,*)$. This follows from the definition of cubical along with the observation, explained in the proof of Theorem~\ref{thm:volumecomputation} (see Equation~\eqref{eq:positivity0}), that the quantity in \eqref{eq:positivity} is equal to the coefficient of $u_\rho$ in the unique expression of $w_\sigma$ as a linear combination in $\{u_\eta \mid \eta\in\sigma(1)\}$.
\end{proof}

\section{Tropical fans and volume polynomials}

In this section, we connect the volume computations of normal complexes to the square-free expression of products in Chow rings, leading to a proof of our main result. The key preliminary fact we require is the following.

\begin{proposition}\label{prop:degreemaps}
Let $\Sigma$ be a simplicial $d$-fan in $N_\R$ and let $\omega:\Sigma(d)\rightarrow\R_{>0}$ be a weight function. Then $(\Sigma,\omega)$ is a tropical fan if and only if there is a well-defined linear degree map
\[
\deg_{\Sigma,\omega}:A^d(\Sigma)\rightarrow\R
\]
satisfying $\deg_{\Sigma,\omega}(X_\sigma)=\omega(\sigma)$ for every $\sigma\in\Sigma(d)$.
\end{proposition}

\begin{proof}
In the unimodular setting, this result is a special case of \cite[Proposition~5.6]{AHK}, and the proof given there generalizes to the simplicial setting. For the reader's convenience, we outline the ideas here.

Define $Z^d(\Sigma)$ to be the vector subspace of $\R[x_\rho\mid\rho\in\Sigma(1)]$ generated by monomials of the form $x_\sigma$ with $\sigma\in\Sigma(d)$. By Theorem~\ref{thm:squarefree}, every element of $A^d(\Sigma)$ can be written as a linear combination of monomials of the form $X_\sigma$ with $\sigma\in\Sigma(d)$, and it follows that
\[
A^d(\Sigma)=\frac{Z^d(\Sigma)}{(\I+\J)\cap Z^d(\Sigma)}.
\]
Define the linear map
\begin{align*}
\deg_{\Sigma,\omega}:Z^d(\Sigma)&\rightarrow\R\\
x_\sigma&\mapsto\omega(\sigma).
\end{align*}
Then $\deg_{\Sigma,\omega}$ descends to the desired tropical degree map on $A^d(\Sigma)$ if and only if it vanishes on all elements of $(\I+\J)\cap Z^d(\Sigma)$. Some moments reflecting should convince the reader that the subspace $(\I+\J)\cap Z^d(\Sigma)$ is generated by polynomials of the form
\[
x_\tau\sum_{\sigma\in\Sigma(d)\atop \tau\prec\sigma}\langle v,u_{\sigma\setminus\tau}\rangle x_{\sigma\setminus\tau}
\]
where $\tau\in\Sigma(d-1)$ and $v\in (N_\R^\tau)^\perp\subseteq M_\R$. Thus, the tropical degree map exists if and only if, for every $\tau\in\Sigma(d-1)$, we have
\begin{equation}\label{eq:tropdegreemapexist}
\sum_{\sigma\in\Sigma(d)\atop \tau\prec\sigma}\langle v,\omega(\sigma)u_{\sigma\setminus\tau}\rangle=0\;\;\;\text{ for all }\;\;\;v\in (N_{\tau,\R})^\perp.
\end{equation}
Notice that \eqref{eq:tropdegreemapexist} is satisfied for all $\tau\in\Sigma(d-1)$ if and only if the tropical balancing condition is satisfied:
\[
\sum_{\sigma\in\Sigma(d)\atop \tau\prec\sigma}\omega(\sigma)u_{\sigma\setminus\tau}\in N_{\tau,\R}\;\;\;\text{ for all }\;\;\;\tau\in\Sigma(d-1).\qedhere
\]
\end{proof}

We can now prove the main result of this paper. Recall that the volume function for a simplicial tropical $d$-fan $(\Sigma,\omega)$ is defined by
\begin{align*}
\Vol_{\Sigma,\omega}:A^1(\Sigma)&\rightarrow\R\\
D&\mapsto\deg_{\Sigma,\omega}(D^d).
\end{align*}

\begin{theorem}\label{thm:main}
If $(\Sigma,\omega)$ is a simplicial tropical $d$-fan in $N_\R$ and $*\in\Inn(N_\R)$ is an inner product, then for any $D=\sum z_\rho X_\rho\in A^1(\Sigma)$, we have
\[
\Vol_{\Sigma,\omega}(D)=\sum_{\sigma\in\Sigma(d)}\omega(\sigma)\det(G_\sigma)\sum_{f\in L(\sigma)}\prod_{j=1}^k\frac{\det(G_{\sigma(f,j),f(j)}(z))}{\det(G_{\sigma(f,j)})},
\]
where all notation is as in Theorem~\ref{thm:volumecomputation}. In particular, if $z\in\cC(\Sigma,*)$ is pseudo-cubical, then
\[
\Vol_{\Sigma,\omega}(D)=\Vol(C_{\Sigma,*}(z);\omega).
\]
\end{theorem}

\begin{proof}
We have
\begin{align*}
\Vol_{\Sigma,\omega}(D)&=\deg_{\Sigma,\omega}(D^d)\\
&=\deg_{\Sigma,\omega}\Bigg(\sum_{\sigma\in\Sigma(d)}\Big(\det(G_\sigma)\sum_{f\in L(\sigma)}\prod_{j=1}^d\frac{\det(G_{\sigma(f,j),f(j)}(z))}{\det(G_{\sigma(f,j)})}\Big)X_\sigma\Bigg)\\
&=\sum_{\sigma\in\Sigma(d)}\omega(\sigma)\det(G_\sigma)\sum_{f\in L(\sigma)}\prod_{j=1}^k\frac{\det(G_{\sigma(f,j),f(j)}(z))}{\det(G_{\sigma(f,j)})},
\end{align*}
where the first equality is the definition of the volume function, the second is an application of Theorem~\ref{thm:squarefree}, and the third is the definition of the tropical degree function. The second statement in the theorem is an immediate application of Theorem~\ref{thm:volumecomputation}.
\end{proof}

\section{Example: Bergman fans of matroids}\label{sec:matroids}

In this final section, we present a rich class of examples of balanced fans arising from matroid theory, called Bergman fans. Our main result is that every Bergman fan of a matroid with arbitrary building set admits an open set of inner products for which the cubical cone is nonempty. Theorem~\ref{thm:main} then provides a geometric interpretation for the volume polynomials of all matroids with respect to arbitrary building sets.

\subsection{Matroids, building sets, and Bergman fans}

A \textbf{matroid} $\sM=(E,\cL)$ consists of a finite set $E$, called the \textbf{ground set}, and a collection of subsets $\cL\subseteq 2^E$, called \textbf{flats}, which satisfy the following two conditions:
\begin{enumerate}
\item if $F_1,F_2$ are flats, then $F_1\cap F_2$ is a flat, and
\item if $F$ is a flat, then every element of $E\setminus F$ is contained in exactly one flat that is minimal among the flats that strictly contain $F$.
\end{enumerate}
For notational simplicity, we assume throughout that all matroids are \textbf{loopless}, meaning that the empty set is a flat. Let $\cL^*$ denote the proper flats of $\sM$---those flats that are neither $\emptyset$ nor $E$.

Given a matroid $\sM=(E,\cL)$, the set $\cL$ is partially ordered by set inclusion. Furthermore, given any subset $S\subseteq E$, it follows from Property (1) that there is a minimal flat containing $S$, called the \textbf{closure} of $S$ and denoted $\cl(S)\in\cL$. Defining the \emph{join} ($\vee$) of two flats to be the closure of their union and the \emph{meet} ($\wedge$) of two flats to be their intersection, it follows from the definitions that the flats $\cL$ form a lattice, called the \textbf{lattice of flats} of $\sM$.

A subset $I\subseteq E$ is called \textbf{independent} if $\cl(I_1)\subset\cl(I_2)$ for any $I_1\subset I_2\subseteq I$. The \textbf{rank} of a subset $S\subseteq E$, denoted $\rk(S)$, is the size of its largest independent subset. The rank of $\sM$ is defined as the rank of $E$. An alternative characterization of the rank of flats is given by lengths of flags. A \textbf{flag} in $\sM$ is an increasing sequence of flats:
\[
\cF=(\emptyset\subset F_1\subset F_2\subset\cdots\subset F_\ell)
\]
The number of nonempty flats in a flag $\cF$ is called the \textbf{length} of the flag, denote $\ell(\cF)$. It can be checked from the above definitions that every maximal flag of flats contained in a flat $F$ has length equal to $\rk(F)$. 

A \textbf{building set on $\sM$} is a subset $\cG\subseteq\cL\setminus\{\emptyset\}$ such that, for any flat $F\in\cL\setminus\{\emptyset\}$ and $\max \cG_{\subseteq F}=\{G_1,\dots,G_k\}$, we have an isomorphism of posets
\begin{align*}
\prod_{i=1}^k[\emptyset,G_i]&\cong [\emptyset, F]\\
(F_1,\dots,F_k)&\mapsto F_1\vee\cdots\vee F_k.
\end{align*}
We assume that all building sets contain $E$ and we set $\cG^*=\cG\cap\cL^*=\cG\setminus\{E\}$. Given a building set $\cG$, a subset $\cN\subseteq\cG$ is called \textbf{nested} if, for any set of pairwise incomparable flats $G_1,\dots,G_\ell\in \cN$ with $\ell\geq 2$, we have $G_1\vee\cdots\vee G_\ell\notin\cG$. Let $\Delta_{\sM,\cG}$ denote the collection of nested sets of $\sM$ with respect to $\cG$, and let $\Delta_{\sM,\cG}^*$ denote the collection of nested sets that do not contain $E$. Since subsets of nested sets are nested, both $\Delta_{\sM,\cG}$ and $\Delta_{\sM,\cG}^*$ naturally have the structure of simplicial complexes. The set $\cL\setminus\emptyset$ is a building set for any matroid $\sM$, which we denote $\cG_{\max}$. With respect to $\cG_{\max}$, it follows from the above definitions that a set of flats is nested if and only if it forms a flag.

Consider the free abelian group $\Z^E$ with basis indexed by $E$. For each subset $S\subseteq E$, define the vector $v_S=\sum_{e\in S}v_e\in\Z^E$. Set $N=\Z^E/\Z v_E$ and for each subset $S\subseteq E$, define $u_S=[v_S]$. The \textbf{Bergman fan of $\sM$ with respect to $\cG$}, denoted $\Sigma_{\sM,\cG}$, is the fan in $N_\R$ with one cone $\sigma_\cN$ indexed by each nested set $\cN\in \Delta_{\sM,\cG}^*$:
\[
\sigma_\cN=\cone(u_G\;|\;G\in \cN).
\]

\begin{example}
Consider the rank $3$ matroid $\sM$ on $E=\{0,1,2,3\}$ with the following lattice of flats (set brackets and commas have been omitted for notational simplicity).
\begin{center}
\begin{tikzpicture}[scale = .5]
\draw[-] (0, 4.7) -- (-4.7, 3.2);
\draw[-] (0, 4.7) -- (-1.9, 3.2);
\draw[-] (0, 4.7) -- (1.9, 3.2);
\draw[-] (0, 4.7) -- (4.7, 3.2);
\node at (0,5.3) {$0123$};
\node at (-5, 2.8) {$01$};
\node at (-2, 2.8) {$02$};
\node at (2, 2.8) {$03$};
\node at (5, 2.8) {$123$};
\draw[-] (-5, 2.3) -- (-5, 0);
\draw[-] (-5, 2.3) -- (-2, 0);
\draw[-] (-2, 2.3) -- (-5, 0);
\draw[-] (-2, 2.3) -- (2, 0);
\draw[-] (2, 2.3) -- (5, 0);
\draw[-] (2, 2.3) -- (-5, 0);
\draw[-] (5, 2.3) -- (5, 0);
\draw[-] (5, 2.3) -- (2, 0);
\draw[-] (5, 2.3) -- (-2, 0);
\node at (-5, -0.4) {$0$};
\node at (-2, -0.4) {$1$};
\node at (2, -0.4) {$2$};
\node at (5, -0.4) {$3$};
\draw[-] (-4.8, -0.8) -- (0, -2.4);
\draw[-] (-1.9, -0.8) -- (0, -2.4);
\draw[-] (1.9, -0.8) -- (0, -2.4);
\draw[-] (4.8, -0.8) -- (0, -2.4);
\node at (0, -2.9) {$\emptyset$};
\end{tikzpicture}
\end{center}
The Bergman fan $\Sigma_{\sM,\cG_{\max}}$ of $\sM$ with respect to the maximal building set $\cG_{\max}$ is depicted in Example~\ref{ex:balancedfan}. The other possible building sets arise from removing some subset of the decomposable flats $\{01,02,03\}$, and the Bergman fans with respect to these building sets are obtained by removing the corresponding subset of the rays $\{\rho_{01},\rho_{02},\rho_{03}\}$ from $\Sigma_{\sM,\cG_{\max}}$.
\end{example}

Bergman fans of matroids with building sets have been studied quite extensively (see, for example, \cite{ArdilaKlivans,FK,FY,FM,FS}). The Bergman fan $\Sigma_{\sM,\cG}$ is unimodular \cite[Proposition~2]{FY}. In addition, given any building set $\cG$, the fan $\Sigma_{\sM,\cG_{\max}}$ can be obtained from $\Sigma_{\sM,\cG}$ by a sequence of stellar subdivisions \cite[Proposition~4.2]{FM}. This fact has two important consequences that are central to our current discussion.
\begin{enumerate}
\item Since $\Sigma_{\sM,\cG_{\max}}$ is pure of dimension $r=\rk(\sM)-1$ (every maximal flag has length $r$), it follows that $\Sigma_{\sM,\cG}$ is pure of dimension $r$ for any building set $\cG$.
\item Since $\Sigma_{\sM,\cG_{\max}}$ is balanced (this follows from the second axiom in the definition of matroids, see \cite[Proposition~3.10]{Huh} for a proof), it then follows from \cite[Lemma~2.11(b)]{AllermanRau} that $\Sigma_{\sM,\cG}$ is also balanced for any building set $\cG$.
\end{enumerate}
Thus, our developments of normal complexes of balanced fans apply in the setting of matroids and Bergman fans. Let
\[
\Vol_{\sM,\cG}\;\;\;\text{ and }\;\;\; C_{\sM,\cG,*}(z),
\]
denote the volume polynomial and the normal complex associated to the Bergman fan $\Sigma_{\sM,\cG}$, where $*\in\Inn(N_\R)$ is any inner product and $z\in\overline\cC(\Sigma_{\sM,\cG},*)$. Theorem~\ref{thm:main} implies that the volume polynomial is computed by
\begin{equation}\label{eq:bergman1}
\Vol_{\sM,\cG}(z)=\sum_{\sigma\in\Sigma_{\sM,\cG}{(r)}}\det(G_\sigma)\sum_{f\in L(\sigma)}\prod_{j=1}^r\frac{\det(G_{\sigma(f,j),f(j)})}{\det(G_{\sigma(f,j)})},
\end{equation}
and that, for any pseudo-cubical value $z\in\overline\cC(\Sigma_{\sM,\cG},*)$, we have
\begin{equation}\label{eq:bergman2}
\Vol_{\sM,\cG}(z)=\Vol(C_{\sM,\cG*}(z)).
\end{equation}

Bergman fans exhibit a great deal of structure, and this structure was recently exploited by Adiprasito, Huh, and Katz \cite{AHK} (for maximal $\cG$) and Ardila, Denham, and Huh \cite{ADH} (for arbitrary $\cG$) to show that matroid Chow rings $A^\bullet(\sM,\cG)=A^\bullet(\Sigma_{\sM,\cG})$ satisfy the K\"ahler package, meaning that they behave in many ways similarly to Chow rings of smooth, projective varieties. In fact, because matroid Chow rings satisfy Poincar\'e duality---which is just one piece of the K\"ahler package---it follows that the volume polynomial $\Vol_{\sM,\cG}(z)$ determines the entire Chow ring $A^\bullet(\sM,\cG)$ \cite[Lemma~13.4.7]{Toric}. In the setting of maximal building sets, volume polynomials have been previously studied, and there are at least two combinatorial formulas for volume polynomials of matroids with respect to $\cG_{\max}$ \cite{Eur,BES,RD}. Equation \eqref{eq:bergman1} provides a continuous family of new formulas for volume polynomials of matroids with arbitrary building sets, one for each choice of inner product.

The initial aim of this work was to introduce volume-theoretic tools into the study of volume polynomials of matroids; in other words, to put the ``volume'' back in ``volume polynomials'' of matroids. In principle, this is accomplished by Equation~\eqref{eq:bergman2}; however, it is not obvious that the hypothesis of \eqref{eq:bergman2} can ever be satisfied. In other words, the cubical hypothesis is a rather restrictive constraint on the choice of $*$ and $z$, and it's not clear that cubical values ever exist. We resolve this issue with the next result.

\begin{proposition}\label{prop:existscubical}
If $\sM=(E,\cL)$ is a matroid and $\cG$ is a building set, then there exists a nonempty open set $U\subseteq\Inn(N_\R)$ such that, for any $*\in U$, we have $\cC(\Sigma_{\sM,\cG},*)\neq\emptyset$. More specifically, if we label the ground set $E=\{e_0,\dots,e_n\}$ and let $*$ be the standard dot product with respect to the basis $u_{e_1},\dots,u_{e_n}\in N_\R$, then there exists a cubical value $z\in\cC(\Sigma_{\sM,\cG},*)$.
\end{proposition}

The proof of this proposition requires one important property of nested sets, which is that any two incomparable elements of a nested set are disjoint. This property can be checked from the definitions above, or a proof can be found in \cite[Section~2]{FK}.

\begin{proof}[Proof of Proposition~\ref{prop:existscubical}]
It follows from the definitions that the existence of a cubical value is an open condition on $\Inn(N_\R)$; therefore, the first statement in the proposition follows from the second. Label the ground set $E=\{e_0,\dots,e_n\}$ and let $*$ be the standard dot product with respect to the basis $u_{e_1},\dots,u_{e_n}\in N_\R$. By definition, note that
\[
u_{e_0}=-\sum_{i=1}^nu_{e_i}.
\]
Choose some $m\gg 0$ and for every $G\in\cG^*$, set
\[
z_G=\begin{cases}
|G|-m^{-|G^c|} &\text{if }e_0\notin G\\
|G^c|-m^{-|G|} &\text{if }e_0\in G.
\end{cases}
\]
We claim that $z\in\cC(\Sigma_{\sM,\cG},*)$. In order to verify this, we must prove that, for each nested set $\cN$, we have $w_\cN=w_{\sigma_\cN}\in\sigma_{\cN}^\circ.$ 

Fix a nested set $\cN$ and write
\[
w_\cN=\sum_{G\in\cN} a_{\cN,G}u_{G}.
\]
We must prove that $a_{\cN,G}>0$ for all $G\in\cN$. The coefficients $ a_{\cN,G}$ are determined by the linear equations
\[
w_\cN *u_{G}=z_{G}\;\;\;\text{ for all }\;\;\;G\in\cN.
\]
In order to write these linear equations more explicitly, notice that, for $G_1,G_2\in\cG^*$, we have
\[
u_{G_1} * u_{G_2}=\begin{cases}
|G_1\cap G_2| &\text{if }e_0\notin G_1\text{ and }e_0\notin G_2;\\
-|G_1\cap G_2^c| &\text{if }e_0\notin G_1\text{ and }e_0\in G_2;\\
|G_1^c\cap G_2^c| &\text{if }e_0\in G_1\text{ and }e_0\in G_2.
\end{cases}
\]

We now fix notation that will be useful in the argument. Let $\cN_0\subseteq\cN$ be the subset of flats containing $e_0$. Since incomparable elements of $\cN$ are disjoint, $\cN_0$ is totally ordered; let $G_0$ denote the minimal element of $\cN_0$. For each $G\in\cN$, let $\widehat G$ be the minimal flat in $\cN_0\cup\{E\}$ that contains $G$. Using this notation, the linear equations defining $w_\cN$ become
\begin{equation}\label{eq:linear1}
\sum_{F\in\cN\atop F\subset G}|F| a_{\cN,F}+|G|\sum_{F\in\cN\atop G\subseteq F\subset \widehat G} a_{\cN,F}-|G|\sum_{F\in\cN\atop G_0\subseteq F\subset\widehat G} a_{\cN,F}=|G|-m^{-|G^c|} \;\;\;\text{ if }\;\;\; G\notin\cN_0
\end{equation}
and
\begin{equation}\label{eq:linear2}
|G^c|\sum_{F\in\cN\atop G_0\subseteq F\subset G} a_{\cN,F}+\sum_{F\in\cN\atop G\subseteq F}|F^c| a_{\cN,F}-\sum_{F\in\cN\setminus\cN_0 \atop G\subset \widehat F}|F| a_{\cN,F}=|G^c|-m^{-|G|} \;\;\;\text{ if }\;\;\; G\in\cN_0.
\end{equation}

For any $G\in\cN$, let $G_+$ denote the minimal element of $\cN\cup\{E\}$ strictly containing $G$. Consider some $G\in\cN_0$ with $G_+\neq E$. If we subtract Equation~\eqref{eq:linear2} for $G_+$ from Equation~\eqref{eq:linear2} for $G$, we obtain the equation
\begin{equation}\label{eq:linear3}
(|G^c|-|G_+^c|)\sum_{F\in\cN\atop G_0\subseteq F\subseteq G} a_{\cN,F}-\sum_{F\in\cN\setminus\cN_0\atop \widehat F=G_+}|F|a_{\cN,F}=|G^c|-|G_+^c|-m^{-|G|}+m^{-|G_+|}.
\end{equation}
Notice that every $F\in\cN\setminus\cN_0$ with $\widehat F=G_+$ is a subset of a unique $H\in\cN\setminus\cN_0$ with $H_+=G_+$. Therefore, summing Equation~\eqref{eq:linear1} for all $H\in\cN\setminus\cN_0$ with $H_+=G_+$, we obtain the equation
\begin{equation}\label{eq:linear4}
\sum_{F\in\cN\setminus\cN_0\atop \widehat F=G_+}|F|a_{\cN,F}-\sum_{H\in\cN\setminus\cN_0\atop H_+=G_+}|H|\sum_{F\in\cN\atop G_0\subseteq F\subseteq G}a_{\cN,F}=\sum_{H\in\cN\setminus\cN_0\atop H_+=G_+}(|H|-m^{-|H^c|}).
\end{equation}
Substituting \eqref{eq:linear4} into \eqref{eq:linear3} and simplifying, it follows that, for any $G\in\cN_0$, we have
\begin{equation}\label{eq:linear5}
\sum_{F\in\cN\atop G_0\subseteq F\subseteq G} a_{\cN,F}=
\begin{cases}
1-\frac{m^{-|G|}-m^{-|G_+|}-\sum_{H\in\cN\setminus\cN_0\atop H_+=G_+}m^{-|H^c|}}{|G^c|-|(G_+)^c|-\sum_{H\in\cN\setminus\cN_0\atop H_+=G_+}|H|} &\text{ if }G_+\neq E\\
1-\frac{m^{-|G|}}{|G^c|} &\text{ if }G_+= E,
\end{cases}
\end{equation}
where the second equation simply follows from \eqref{eq:linear2} applied to the unique maximal element $G\in\cN_0$ with $G_+=E$.

Since $\cN$ is a nested set, it follows that the denominator in \eqref{eq:linear5} is always a positive integer; more specifically, this follows from the fact that $\{G\}\cup\{H\in\cN\setminus\cN_0\mid H_+=G_+\}$ is a collection of pairwise incomparable elements of $\cN$ that are all subsets of $G_+$, so they must be pairwise disjoint subsets of $G_+$, and their union cannot be all of $G_+$ or else their join would be equal to $G_+\in\cG$, contradicting the nested condition. For $m\gg 0$, notice that the leading term in the quotient in \eqref{eq:linear5}  is $-m^{-|G|}$, from which it follows that the right-hand side of \eqref{eq:linear5} is increasing with respect to $G$. Thus, taking successive differences to solve for each $a_{\cN,G}$, we conclude that $ a_{\cN,G}>0$ for all $G\in\cN_0$ with $G\neq G_0$. For $G=G_0$, notice that the right-hand side of \eqref{eq:linear5} is positive for $m\gg 0$ simply because the quotient is very small, implying that $a_{\cN,G_0}>0$. Thus, we conclude that $a_{\cN,G}>0$ for all $G\in\cN_0$.

Suppose now that $G\notin\cN_0$. Then taking Equation~\eqref{eq:linear1} for $G$ and subtracting from it Equation~\eqref{eq:linear1} for all $H\in\cN$ with $H_+=G$, we obtain the equation
\[
\Big(|G|-\sum_{H\in\cN\atop H_+=G}|H|\Big)\Big(\sum_{F\in\cN\atop G\subseteq F\subset\widehat G} a_{\cN,F}-\sum_{F\in\cN\atop G_0\subseteq F\subset \widehat G}a_{\cN,F}\Big)=|G|-m^{-|G^c|}-\sum_{H\in\cN\atop H_+=G}\big(|H|-m^{-|H^c|}\big).
\]
 Simplifying, we may write
 \begin{equation}\label{eq:linear6}
 \sum_{F\in\cN\atop G\subseteq F\subset\widehat G} a_{\cN,F}=\Big(1+\sum_{F\in\cN\atop G_0\subseteq F\subset \widehat G}a_{\cN,F}\Big)-\frac{m^{-|G^c|}-\sum_{H\in\cN\atop H_+=G}m^{-|H^c|}}{|G|-\sum_{H\in\cN\atop H_+=G}|H|}
 \end{equation}
 
As in the previous case, the denominator in the second term of the right-hand side of Equation~\eqref{eq:linear6} is positive because $\cN$ is a nested set.  For $m\gg 0$, notice that $-m^{-|G^c|}$ is the leading term of the quotient in Equation~\eqref{eq:linear6}, from which it follows that the right-hand side of \eqref{eq:linear4} is decreasing with respect to $G$. Therefore, taking successive differences to solve for each $a_{\cN,G}$, we see that $a_{\cN,G}>0$ for all $G\notin\cN$ with $G_+\neq\widehat G$. In the case that $G_+=\widehat G$, then $ a_{\cN,G}$ is the only term in the left-hand side of \eqref{eq:linear6}, and the fact that $a_{\cN,F}>0$ for all $F\in\cN_0$, which is what we already argued above, then implies that the right-hand side of \eqref{eq:linear6} is positive for $m\gg 0$. Thus, we conclude that $a_{\cN,G}>0$ for all $G\notin\cN_0$, finishing the proof of the proposition.
\end{proof}

\bibliographystyle{alpha}

\begin{thebibliography}{BHM{\etalchar{+}}22}

\bibitem[ADH20]{ADH}
F.~Ardila, G.~Denham, and J.~Huh.
\newblock Lagrangian geometry of matroids.
\newblock Preprint: arXiv:2004.13116, 2020.

\bibitem[AHK18]{AHK}
K.~Adiprasito, J.~Huh, and E.~Katz.
\newblock Hodge theory for combinatorial geometries.
\newblock {\em Ann. of Math. (2)}, 188(2):381--452, 2018.

\bibitem[AK06]{ArdilaKlivans}
F.~Ardila and C.~J. Klivans.
\newblock The {B}ergman complex of a matroid and phylogenetic trees.
\newblock {\em J. Combin. Theory Ser. B}, 96(1):38--49, 2006.

\bibitem[AP20]{AP1}
O.~Amini and M.~Piquerez.
\newblock Hodge theory for tropical varieties.
\newblock Preprint: arXiv:2007.07826, 2020.

\bibitem[AP21]{AP2}
O.~Amini and M.~Piquerez.
\newblock Homology of tropical fans.
\newblock Preprint: arXiv:2105.01504, 2021.

\bibitem[AR10]{AllermanRau}
L.~Allermann and J.~Rau.
\newblock First steps in tropical intersection theory.
\newblock {\em Math. Z.}, 264(3):633--670, 2010.

\bibitem[BDCP90]{BDP}
E.~Bifet, C.~De~Concini, and C.~Procesi.
\newblock Cohomology of regular embeddings.
\newblock {\em Adv. Math.}, 82(1):1--34, 1990.

\bibitem[BES20]{BES}
S.~Backman, C.~Eur, and C.~Simpson.
\newblock Simplicial generation of {C}how rings of matroids.
\newblock {\em S\'{e}m. Lothar. Combin.}, 84B:Art. 52, 11, 2020.

\bibitem[BHM{\etalchar{+}}20]{BHMPW2}
T.~Braden, J.~Huh, J.~P. Matherne, N.~Proudfoot, and B.~Wang.
\newblock Singular hodge theory for combinatorial geometries.
\newblock Preprint: arXiv:2010.06088, 2020.

\bibitem[BHM{\etalchar{+}}22]{BHMPW1}
T.~Braden, J.~Huh, J.~P. Matherne, N.~Proudfoot, and B.~Wang.
\newblock A semi-small decomposition of the {C}how ring of a matroid.
\newblock {\em Adv. Math.}, 409(part A):Paper No. 108646, 49, 2022.

\bibitem[Bri96]{Brion}
M.~Brion.
\newblock Piecewise polynomial functions, convex polytopes and enumerative
  geometry.
\newblock In {\em Parameter spaces ({W}arsaw, 1994)}, volume~36 of {\em Banach
  Center Publ.}, pages 25--44. Polish Acad. Sci. Inst. Math., Warsaw, 1996.

\bibitem[CLS11]{Toric}
D.~A. Cox, J.~B. Little, and H.~K. Schenck.
\newblock {\em Toric varieties}, volume 124 of {\em Graduate Studies in
  Mathematics}.
\newblock American Mathematical Society, Providence, RI, 2011.

\bibitem[Dan78]{Danilov}
V.~I. Danilov.
\newblock The geometry of toric varieties.
\newblock {\em Uspekhi Mat. Nauk}, 33(2(200)):85--134, 247, 1978.

\bibitem[DCP95a]{DP1}
C.~De~Concini and C.~Procesi.
\newblock Hyperplane arrangements and holonomy equations.
\newblock {\em Selecta Math. (N.S.)}, 1(3):495--535, 1995.

\bibitem[DCP95b]{DP2}
C.~De~Concini and C.~Procesi.
\newblock Wonderful models of subspace arrangements.
\newblock {\em Selecta Math. (N.S.)}, 1(3):459--494, 1995.

\bibitem[DR22]{RD}
J.~Dastidar and D.~Ross.
\newblock Matroid psi classes.
\newblock {\em Selecta Math. (N.S.)}, 28(3):Paper No. 55, 38, 2022.

\bibitem[Eur20]{Eur}
C.~Eur.
\newblock Divisors on matroids and their volumes.
\newblock {\em J. Combin. Theory Ser. A}, 169:105135, 31, 2020.

\bibitem[FK04]{FK}
E.-M. Feichtner and D.~N. Kozlov.
\newblock Incidence combinatorics of resolutions.
\newblock {\em Selecta Math. (N.S.)}, 10(1):37--60, 2004.

\bibitem[FM05]{FM}
E.~M. Feichtner and I.~M\"{u}ller.
\newblock On the topology of nested set complexes.
\newblock {\em Proc. Amer. Math. Soc.}, 133(4):999--1006, 2005.

\bibitem[FS05]{FS}
E.~M. Feichtner and B.~Sturmfels.
\newblock Matroid polytopes, nested sets and {B}ergman fans.
\newblock {\em Port. Math. (N.S.)}, 62(4):437--468, 2005.

\bibitem[FY04]{FY}
E.~M. Feichtner and S.~Yuzvinsky.
\newblock Chow rings of toric varieties defined by atomic lattices.
\newblock {\em Invent. Math.}, 155(3):515--536, 2004.

\bibitem[Her72]{Heron}
A.~P. Heron.
\newblock Matroid polynomials.
\newblock In {\em Combinatorics ({P}roc. {C}onf. {C}ombinatorial {M}ath.,
  {M}ath. {I}nst., {O}xford, 1972)}, pages 164--202, 1972.

\bibitem[Huh18]{Huh}
J.~Huh.
\newblock Tropical geometry of matroids.
\newblock In {\em Current developments in mathematics 2016}, pages 1--46. Int.
  Press, Somerville, MA, 2018.

\bibitem[McM89]{McMullen1}
P.~McMullen.
\newblock The polytope algebra.
\newblock {\em Adv. Math.}, 78(1):76--130, 1989.

\bibitem[McM93]{McMullen2}
P.~McMullen.
\newblock On simple polytopes.
\newblock {\em Invent. Math.}, 113(2):419--444, 1993.

\bibitem[NOR23]{MNR}
L.~Nowak, P.~O'{M}elveny, and D.~Ross.
\newblock Mixed volumes of normal complexes.
\newblock Preprint: arXiv:2301.05278, 2023.

\bibitem[Rot71]{Rota}
G.-C. Rota.
\newblock Combinatorial theory, old and new.
\newblock In {\em Actes du {C}ongr\`es {I}nternational des {M}ath\'{e}maticiens
  ({N}ice, 1970), {T}ome 3}, pages 229--233. 1971.

\bibitem[Sch20]{Schneider2}
Rolf Schneider.
\newblock On a formula for the volume of polytopes.
\newblock In {\em Geometric aspects of functional analysis. {V}ol. {II}},
  volume 2266 of {\em Lecture Notes in Math.}, pages 335--345. Springer, Cham,
  [2020] \copyright 2020.

\bibitem[Wel76]{Welsh}
D.~J.~A. Welsh.
\newblock {\em Matroid theory}.
\newblock Academic Press [Harcourt Brace Jovanovich, Publishers], London-New
  York, 1976.
\newblock L. M. S. Monographs, No. 8.

\end{thebibliography}

\newcommand{\etalchar}[1]{$^{#1}$}

\end{document}